\documentclass[a4paper, 11pt]{amsart}
\usepackage{amssymb}
\usepackage{amsmath}
\usepackage{amsthm}
\usepackage{amscd}
\usepackage{amsfonts,mathrsfs}
\usepackage{hyperref}
\usepackage{a4wide}
\usepackage{color,todonotes}

\newcounter{tmpcounter}

\numberwithin{equation}{section}
\numberwithin{excounter}{section}
\numberwithin{defcounter}{section}
\numberwithin{thmcounter}{section}
\numberwithin{corcounter}{section}

\newtheorem{prop}{Proposition}[section]
\newtheorem{cor}[prop]{Corollary}
\newtheorem{ex}[prop]{Example}
\newtheorem{lem}[prop]{Lemma}
\newtheorem{thm}[prop]{Theorem}

\newtheorem{defin}[prop]{Definition}

\newtheorem{rem}[prop]{Remark}

\newcommand{\Jac} {\mathop{\mathrm{Jac}}}
\newcommand{\Tr} {\mathop{\mathrm{Tr}}}

\newcommand{\disc} {\mathop{\mathrm{disc}}}

\newcommand{\Ree} {\mathop{\mathrm{Re}}}

\newcommand{\degr} {\mathop{\mathrm{deg}}}

\newcommand{\Spec} {\mathop{\mathrm{Spec}}}

\newcommand{\Lie} {\mathop{\mathrm{Lie}}}
\newcommand{\Pic} {\mathop{\mathrm{Pic}}}

\newcommand{\IH}{\mathbb H}

\newcommand{\IP}{\mathbb P}
\newcommand{\IC}{\mathbb C}
\newcommand{\IZ}{\mathbb Z}
\newcommand{\IQ}{\mathbb Q}
\newcommand{\IR}{\mathbb R}
\newcommand{\IQbar}{\overline{\mathbb Q}}

\newcommand{\SSL}[1]{{\rm SL}_{#1}}
\newcommand{\SL}[2]{\SSL{#1}({#2})}
\newcommand{\SP}[2]{{\rm Sp}_{#1}({#2})}

\newcommand{\GL}[2]{{\rm GL}_{#1}({#2})}
\newcommand{\GLp}[2]{{\rm GL}^{+}_{#1}({#2})}

\newcommand{\cO}{\mathcal{O}}
\renewcommand{\O}[1]{\mathcal{O}_{#1}}

\newcommand{\Endo}[1]{{\rm End}({#1})}

\newcommand{\aGal}[1]{{\rm Gal}(\overline{#1}/#1)}
\newcommand{\gal}[1]{{\rm Gal}({#1})}

\newcommand{\atopx}[2]{\genfrac{}{}{0pt}{}{#1}{#2}}

\newcommand{\trans}[1]{{}^t{#1}}

\def\ssm{\smallsetminus}

\newcommand{\reldisc}[1]{\mathfrak{d}_{#1}}

\newcommand{\diff}[1]{\mathscr{D}_{#1}}

\newcommand{\norm}[2]{{\mathrm N}_{#2}({#1})}
\newcommand{\inorm}[1]{{\mathrm N}({#1})}

\newcommand{\mat}[2]{\mathrm {Mat}_{#2}({#1})}

\newcommand{\imagS}{\mathop{\mathrm{Im}}}
\newcommand{\realS}{\mathop{\mathrm{Re}}}

\newcommand{\imag}[1]{\imagS({#1})}
\newcommand{\real}[1]{\realS({#1})}

\newcommand{\diag}[1]{{\mathrm{diag}}({#1})}

\newcommand{\tr}[2]{{\rm Tr}_{#2}({#1})}

\newcommand{\cg}[1]{{Cl}_{#1}}

\newcommand{\sfd}{\mathcal{F}} 
\newcommand{\hfd}{\mathcal{F}} 

\newcommand{\res}[1]{{\rm Res}_{#1}}

\newcommand{\jac}[1]{{\rm Jac}({#1})}

\newcommand{\aut}[1]{{\rm Aut}({#1})}

\newcommand{\art}{{\rm art}}

\newcommand{\magicsection}{\eta}

\newcommand{\unr}[1]{{#1}^{unr}}

\renewcommand{\subset}{\subseteq}

\begin{document}

\renewcommand{\theenumi}{(\roman{enumi})}

\title[Bad reduction and CM jacobians]{Bad reduction of genus $2$ curves with CM jacobian varieties}
 

\author[Philipp {\sc Habegger}]{{\sc Philipp} Habegger}
\address{Philipp {\sc Habegger}\\
Departement Mathematik und Informatik\\
Spiegelgasse 1\\
4051 Basel \\
Switzerland
}
\email{philipp.habegger@unibas.ch}

\author[Fabien {\sc Pazuki}]{{\sc Fabien} Pazuki}
\address{Fabien {\sc Pazuki}\\
Department of Mathematical Sciences\\
University of Copenhagen\\
Universitetsparken 5\\
2100 Copenhagen, Denmark
and\\
IMB, Univ. Bordeaux\\
351, cours de la Lib\'eration\\
33405 Talence, France\\}
\email{fabien.pazuki@math.u-bordeaux.fr}
\urladdr{http://www.math.u-bordeaux.fr/~fpazuki}

\maketitle

%

%



{\small{\textsc{Abstract}: 
We show that a genus $2$ curve over a number field whose jacobian
has complex multiplication will usually have stable bad reduction at
some prime.
We prove this by computing the Faltings height of the jacobian in two
different ways.
First, we use a formula by Colmez and Obus specific to the CM case and valid
when the CM field is an abelian extension of the rationals.
This formula links the height and the logarithmic derivatives of an
$L$-function. 
The second formula involves a decomposition of the height into local
terms based on a hyperelliptic model.
We use results of Igusa, Liu, and Saito to show that the
  contribution at the finite places in our decomposition measures the stable bad reduction
of the curve and
subconvexity bounds by Michel and Venkatesh together with an 
  equidistribution result of Zhang to handle the infinite places.}}

\tableofcontents

{\flushleft
}


\section{Introduction}





By a  curve we mean a smooth, geometrically connected, projective curve $C$ defined
over a field $k$. 
Its
 jacobian variety $\jac{C}$
 is a principally polarised abelian variety defined over $k$.
For any abelian variety $A$ defined over $k$ 
we  write $\Endo{A}$ for the ring of geometric
endomorphisms of $A$, \textit{i.e.} the ring of endomorphisms of the base
change of  $A$ to a given algebraic closure of $k$. 
For brevity we say that $A$ has CM if its base change to an algebraic
closure of $k$ has complex multiplication and if $k$ has
characteristic $0$.
We also say that $C$ has CM if $\jac{C}$ does. 
A curve defined over $\IQbar$ is said to have good reduction everywhere if it has
potentially good reduction at all finite places of
 a number field over which it is defined. 

By the work of Serre and Tate \cite{SerTat},  an abelian variety defined over a
number field with CM has potentially good reduction at all finite places. 
If a curve of positive genus which is defined over a
number field has good reduction at a given finite place, then so does
its jacobian variety. However, the converse statement is
false already in the genus $2$ case,  \textit{cf.} entry 
[$I_0\text{-}I_0\text{-}m$] in Namikawa and Ueno's
classification table 
\cite{NamikawaUeno} in equicharacteristic $0$. The main result of our
paper, which we discuss in greater detail below, states that this phenomenon prevails
for certain families of CM curves of genus $2$.

\begin{thm}
\label{thm:finiteness}
Let $F$ be a  real quadratic number field.
Up-to isomorphism there are only finitely many
curves $C$ of genus $2$ defined over $\IQbar$
with good reduction everywhere 
and such that $\Endo{\jac{C}}$ is the maximal order of a quartic,
cyclic, totally imaginary number field containing $F$. 
\end{thm}

This finiteness result is of a familiar type for objects in arithmetic
geometry.
        A number field has only finitely many unramified extensions of given 
degree due to the  Theorem of Hermite-Minkowski.  The Shafarevich Conjecture,
 proved by Faltings \cite{Faltings:ES}, ensures that again there are
 only finitely many curves defined over a fixed number field, of fixed
 positive genus, with good reduction outside a fixed finite set of
 places. 
Fontaine  \cite[page 517]{Font85}  proved that there is no non-zero
abelian variety of any dimension with good reduction at all finite places if one fixes the field
of definition  to be either $\mathbb{Q}$, $\mathbb{Q}(i)$,
$\mathbb{Q}(i\sqrt{3})$ or $\mathbb{Q}(\sqrt{5})$. 
In particular, there exists no curve over $\IQ$ of positive genus that
has good reduction at all primes. 
 Schoof obtained    finiteness results along these lines for
 certain additional cyclotomic fields  \cite{Schoof03}.

Let us stress here that there are
 infinitely many curves of genus $2$ defined over $\IQbar$ with good reduction everywhere. One can deduce this fact from 
Moret-Bailly's Exemple 0.9 \cite{MB:Skolem}.

Our result does not seem to be a direct consequence of the theorems
mentioned above.
Instead of working over a fixed number field our finiteness result
concerns curves over the algebraically closed field $\IQbar$. 
 Indeed, it is not possible to uniformly bound the degree over $\IQ$
 of a curve of genus $2$  whose
jacobian variety has complex multiplication.


\begin{ex}
Let us exhibit an infinite family of
genus $2$ curves
with CM such that the endomorphism ring is the ring of algebraic integers
in a cyclic extension of $\IQ$ that contains
 $\IQ(\sqrt{5})$.

Suppose  $p\equiv 1 \mod 12$ is a prime, then 
\begin{equation*}
f = x^4 + 10px^2+5p^2
\end{equation*}
has roots
\begin{equation}
\label{eq:froots}
\pm \sqrt{ -p(5 \pm 2 \sqrt{5})}.
\end{equation}
So the splitting field $K=K_p$ of $f$ over $\IQ$ is a CM-field with
maximal totally real subfield 
$\IQ(\sqrt{5})$. The product of  two roots  lies in
$\IQ(\sqrt 5)$, so $K/\IQ$ is a Galois extension. 
But such a product does not lie in $\IQ$, so
the Galois group of $K/\IQ$ is not isomorphic to $(\IZ/2\IZ)^2$. 
 This
means that $K/\IQ$ is a cyclic extension.

 By (\ref{eq:froots})  $K$ is ramified above 
 $p$, so we obtain infinitely many fields $K$ as there are
infinitely many admissible $p$.

There exists a principally polarised abelian surface whose
endomorphism ring is the ring of algebraic integers in $K$,
see for example the paragraph after the proof of van Wamelen's Theorem 4 \cite{vanWamelen}.
As in our situation this abelian
variety is necessarily simple, \textit{cf.} Lemma \ref{lem:reflexlb1} below, the
endomorphism ring must be equal to the ring of integers in $K$.
A principally polarised abelian surface that is not a product of
elliptic curves with the product polarisation is the jacobian of a
 curve of genus $2$
by Corollary 11.8.2 \cite{CAV}.
Therefore, there is a curve  $C=C_p$
defined over a number field such that $\jac{C}$ has complex
multiplication by the ring of algebraic integers in $K$. According to
Theorem \ref{thm:finiteness},
 the curve $C$ has potentially good reduction everywhere
for at most finitely many $p$.

We set $m=(p-1)/12\in\IZ$ and observe that
\begin{equation*}
 2^{-4} f(2x+1) = x^4 + 2x^3 + (30m + 4)x^2 + (30m + 3)x + 45m^2 + 15m + 1
\end{equation*}
 is irreducible modulo $2$ and
modulo $3$. This implies that $K/\IQ$ is unramified above these
primes
and even that they are inert in $K$.
We may apply Goren's Theorem
1 \cite{Goren97} to see  that the semi-stable reduction of $\jac{C}$ at all
all places above $2$ and $3$ is isogenous but not isomorphic to a product of
supersingular elliptic curves. By the paragraph before Proposition
2 \cite{Liu:stables} the  curve $C$ has potentially  good reduction at places
above $2$ and $3$. So bad reduction is not a consequence of  the obstruction
described
by Ibukiyama, Katsura, and Oort's Theorem 3.3(III) \cite{IKO},
\textit{cf.} Goren and Lauter's comment of page 477 \cite{GoLau}.
\end{ex}

The proof of Theorem \ref{thm:finiteness}
 relies heavily on various aspects of the stable Faltings height $h(A)$ of an
 abelian variety $A$ defined over a number field. Indeed, it follows by computing the said height of $\jac{C}$ in two
 different ways if $C$ is a genus $2$ curve defined over $\IQbar$. 
We will be able to bound one of these expressions from
 below and the other one from above. The resulting inequality will
 yield Theorem \ref{thm:discbound} below, a more precise version of our result above.

The first expression of the Faltings height of $\jac{C}$  uses the
 additional hypothesis that 
$C$ has CM as in Theorem \ref{thm:finiteness}. We will use
 Colmez's Conjecture, a theorem in our case due to Colmez \cite{ColmezAnnals}
 and Obus \cite{Obus} as the CM-field $K$ is an abelian extension of
 $\IQ$.  
It enables us to express $h(\jac{C})$ in terms of the logarithmic
 derivative of an $L$-function. 
Using this presentation, Colmez \cite{Colmez} found a lower bound for the
 Faltings height of an elliptic curve with CM when the endomorphism
 ring is a maximal order. The
 bound grows
  logarithmically in the discriminant of the CM-field. 
 We recall
 that the discriminant $\Delta_K$ of 
 $K$  is a
 positive integer as $K$ is a quartic CM-field. 
In our case we obtain a lower bound which is linear in  
 $\log \Delta_K$.
Let $B$ be a real number. So by
 the Theorem of Hermite-Minkowski
 there are only finitely many possibilities for $K$ up-to isomorphism if $h(\jac{C})\le
 B$. 
In the situation of Theorem \ref{thm:finiteness}, the endomorphism
 ring of $\jac{C}$ is the
 maximal order of $K$. So there are only finitely many possibilities
 for  $\jac{C}$ 
up-to isomorphism for fixed $B$.   Torelli's
 Theorem  will imply that there
are at most finitely many possibilities for $C$ up-to isomorphism.

Our  theorem   would follow if we could establish a uniform height
upper bound $B$ as before. 
We were not able to do this directly. 
Instead, we will show that for any
$\epsilon >0$ there is a constant $c(\epsilon,F)$ with
\begin{equation}
\label{eq:wtshJC}
h(\jac{C})\le \epsilon \log \Delta_K + c(\epsilon,F),
\end{equation}
with $F$ the maximal
totally real subfield  of $K$.
For small $\epsilon$ this upper bound is strong enough to compete with the logarithmic
lower bound coming from Colmez's Conjecture because $F$ is fixed in
our Theorem \ref{thm:finiteness}. 

The  upper bound requires the
second expression for the Faltings height of $\jac{C}$ alluded to
above. We still work with a curve $C$ of genus $2$ defined over 
$\IQbar$, but now do not require that $\jac{C}$ has CM. 
Suppose $C$ is the base change to $\IQbar$ of a curve $C_k$ 
 defined over a number field $k\subset\IQbar$. If $C_k$ 
has good reduction
at all places above $2$, then
Ueno \cite{Ueno} decomposed $h(\jac{C})$ 
into a sum over all places of $k$.\footnote{For other explicit formulas, the reader may consult
Autissier's Theorem 5.1 page 1457 of \cite{Aut06} or the second-named
author's Theorems 1.3 and 1.4 of \cite{Paz}.} 
 We present another expression for the Faltings height in 
Theorem \ref{hyperelliptic} by decomposing it into local terms. 
In contrast to Ueno's formula and with our application in mind, we
require that $\jac{C_k}$ has good reduction at all finite places but in turn allow 
$C_k$ to have bad reduction above $2$. 
Our proof of Theorem \ref{hyperelliptic} makes use of the reduction
theory of genus $2$ curves as developed  by Igusa \cite{Igusa:60} and
later by Liu \cite{Liu:stables,Liu:cond}
as well as Saito's generalisation \cite{Saito:Duke88}  of Ogg's formula for
the conductor of an elliptic curve. 
In our  decomposition of $h(\jac{C})$ into  local terms a non-zero
contribution at a
finite place indicates  that  the \emph{curve}
$C_k$ has bad stable reduction at the said place.
In other words, if $C$ has  good reduction everywhere, as
in Theorem \ref{thm:finiteness}, then the finite places do not
contribute to $h(\jac{C})$.
We will also  express the local contribution in $h(\jac{C})$ at the finite places
in terms of the classical Igusa invariants attached to $C$. 

The terms at the archimedean places in Theorem \ref{hyperelliptic} are expressed using a
Siegel modular cusp form  of degree $2$ and weight $10$.
We must bound these infinite places from above in order to arrive at
(\ref{eq:wtshJC}). 
One issue is that the archimedean local term has  a
logarithmic singularity along the divisor where the cusp form
vanishes. This vanishing locus corresponds to the principally
polarised abelian surfaces that are isomorphic to a product of
elliptic curves with the product polarisation. 
The jacobian variety
of a genus $2$ curve defined over $\IC$ is never such a product. 
So in our application, we are never  on the logarithmic singularity.

To obtain the upper bound for $h(\jac{C})$ we must ensure
first that not too many period matrices coming from the conjugates of $\jac{C}$
are  close to the logarithmic singularity. 
 Second, we must show that no period matrix is excessively close to
 the said singularity. 

To achieve the first goal we require Zhang's  Equidistribution Theorem
\cite{zhang:equicm} for 
Galois orbits of CM points on Hilbert
modular surfaces.
 Zhang's result relies on
the powerful subconvexity estimate due to Michel-Venkatesh \cite{MV};
Cohen \cite{Cohen:equi} and Clozel-Ullmo \cite{CU:2005} have related
equidistribution results.
Roughly speaking, equidistribution guarantees that only a small
proportion of period matrices coming from  the Galois orbit of $\jac{C}$ lie close to the
problematic divisor. 

However, equidistribution does not  rule out the
possibility that some period matrix is excessively close to the singular
locus.
To handle this contingency we use the following
 simple but crucial observation.  Inside Siegel's
fundamental domain, the divisor consists of diagonal period matrices
\begin{equation*}
\left(
\begin{array}{cc}
* & 0 \\ 0 & *
\end{array}\right). 
\end{equation*}
A period matrix lying close to this divisor has small off-diagonal
entries.  It is a classical fact that the  period matrix of a CM
abelian variety is algebraic.
Moreover, the degree over $\IQ$ of each entry is bounded from above
in terms of the dimension of the abelian variety. 
 We will use Liouville's inequality to bound the modulus of the
 off-diagonal entries from below.
This enables us to handle the contribution coming from the vanishing
 locus of the
cusp form.

The archimedean contribution to the Faltings height of $\jac{C}$ is
also unbounded near the cusp in Siegel upper half-space. We will again use
the subconvexity estimates to control this contribution on average. 

These various estimates combine to (\ref{eq:wtshJC}). The quantitative
nature of our approach allows for the following quantitative 
estimate which implies Theorem \ref{thm:finiteness}, as we will
see.
We will measure the amount of bad stable reduction of a curve $C_k$ of genus
$2$ defined over a number field $k$ using the 
minimal discriminant  $\Delta^0_{\rm min}(C)$ 
in the sense of Definition \ref{Liu}. It is a non-zero  ideal in  the
ring of integers of $k$ and  $\inorm{\Delta^0_{\rm min}(C)}$ denotes
its norm below.

\begin{thm}
\label{thm:discbound}
Let $F$ be a real quadratic number field.
There exists a constant $c(F)>0$ with the following property. 
Let $C$ be a curve of genus $2$ defined over $\overline{\mathbb{Q}}$
such 
that  $\Endo{\jac{C}}$ is the maximal order of an imaginary quadratic
extension $K$ of $F$ with $K/\IQ$ cyclic.
Then $C$ is the base change to $\IQbar$ of a curve $C_k$ defined over a number field
$k\subset\IQbar$ with
\begin{equation}
\label{eq:DiscKlb}
\log  \Delta_K\leq
c(F)\left(1+\frac{1}{[k:\mathbb{Q}]}\log \inorm{\Delta^0_{\rm min}(C_k)}\right),
\end{equation}
where the normalised norm on the right is invariant under finite field
extensions of $k$. 
\end{thm}

The choice of $k$ will be made during the proof. 
In Theorem \ref{hyperelliptic}(ii) we will be able to
express the normalised norm in terms of the Igusa invariants
of the curve $C$.

Theorem \ref{thm:discbound} implies finiteness results to more general
families than curves with potentially good reduction everywhere. Indeed, an
analog of Theorem \ref{thm:finiteness} is obtained for any collection
where the normalised norm of $\Delta^0_{\rm min}(C_k)$ is uniformly
bounded from above. 

Let $K$ be a quartic CM-field  that is
not bi-quadratic. 
Goren and Lauter \cite{GorenLauter:06} call
 a rational prime $p$ \textit{evil}
for $K$ if there is a principally polarised
 abelian variety with CM by the maximal order of $K$ 
whose reduction over a place  above $p$ is a product of two
supersingular elliptic curves with the product polarisation.
This corresponds to  a genus $2$ curve whose semi-stable
 reduction is bad at a place above $p$ and whose jacobian variety has CM
by the maximal order of $K$. Goren and Lauter proved that evilness
prevails by showing that a given
prime is evil for infinitely many $K$ containing a fixed real
quadratic field with trivial narrow-class group. 
In our Theorem \ref{thm:finiteness} the prime $p$ varies; using Goren and
Lauter's terminology we can restate our result as follows. For all but finitely many
quartic and cyclic CM number fields containing a given real quadratic
field there is an evil prime.

Let us now recall the fundamental result of Deligne and Mumford of \cite{DelMum}, Theorem 2.4 page 89.

\begin{thm}(Deligne-Mumford)
Let $k$ be a field with a discrete valuation and with algebraically closed
residue field. Let $C$ be a  curve over $k$ of genus at least $2$. Then the
jacobian variety $\Jac(C)$ has semi-stable reduction if and only if $C$ has semi-stable reduction.
\end{thm}

The reader should keep in mind that even though a curve and its
jacobian variety
have 
semi-stable reduction simultaneously, it does not mean that the type of reduction (good or bad) is the same.

We conclude this introduction by posing some questions related to our
results and to $\mathcal{A}_g$, the coarse moduli space of principally
polarised abelian varieties of dimension $g\ge 1$.

The authors  conjecture
that there are only finitely many curves $C$ of genus $2$
 defined over $\IQbar$ which have good reduction everywhere and for
 which
$\jac{C}$ has complex multiplication by an order containing the ring
of integers of $F$. 

Our  restriction in Theorem \ref{thm:finiteness} that $K/\IQ$ is abelian  reflects the 
current status of 
 Colmez's
Conjecture. This conjecture is open for general quartic extensions of
 $\IQ$. However, Yang \cite{Yang10} has proved some non-abelian cases
for quartic CM-fields. 

Nakkajima-Taguchi \cite{NT} compute the Faltings height of an
elliptic curve with complex multiplication by a general order. 
They reduce the computation to the case of a maximal order which is
 covered by the Chowla-Selberg formula. As far as the
 authors know, no analog reduction is  known in dimension $2$. 

Our approach relies heavily on
equidistribution of Galois orbits on Hilbert modular surfaces. For this
reason we must fix the maximal total real subfield in our theorem.
However, it is natural to 
ask if the finiteness statement in Theorem \ref{thm:finiteness} holds
without fixing $F$. 
For example,
is the set of points in $\mathcal{A}_2$ consisting of 
 jacobians of curves 
defined over $\IQbar$ with CM and with good reduction everywhere
Zariski non-dense in $\mathcal{A}_2$? One could even  speculate whether
this set is finite.

In genus $g=3$ the image of the Torelli morphism again dominates 
$\mathcal{A}_3$.
 Here too this image contains infinitely jacobian varieties with CM.
 So we ask whether 
the set of CM points 
that come from genus $3$ curves with good reduction everywhere
 is Zariski non-dense in
 $\mathcal{A}_3$ or perhaps even finite. A simplified variant of this
question would ask for non-denseness or finiteness under the restriction that the CM-field
contains a fixed totally real cubic subfield.
Hyperelliptic curves of genus $3$ do not lie Zariski dense in the
moduli space of genus $3$ curves. Thus   a statement like
 Theorem \ref{hyperelliptic} for non-hyperelliptic curves would be necessary. This 
 would be interesting in its own right. 

Starting from genus $g= 4$ it is no longer true that the Torelli morphism 
dominates $\mathcal{A}_4$. The Andr\'e-Oort Conjecture, which is known
unconditionally in this case by work of Pila and Tsimerman \cite{PilaTsimerman:AxLAg}, yields  an additional
 obstruction for a curve of genus $4$ to have CM.
Coleman conjectured that there are only finitely many curves
of fixed genus $g\ge 4$ with CM. Although this conjecture is known to
be false if $g=4$ and $g=6$ by work of de Jong and
Noot \cite{deJongNoot:91}. In any case, a version of Theorem \ref{thm:finiteness} for higher genus curves is entangled with other problems in
arithmetic geometry.

In genus $g=1$ no finiteness result such as Theorem \ref{thm:finiteness}
can hold true, as an elliptic curve with complex multiplication has
potentially good reduction at all finite places. 
However, the first-named author 
proved  \cite{hab:junit} the following finiteness result which is reminiscent
of the current work. Up-to $\IQbar$-isomorphism there are
only finitely many elliptic curves with complex multiplication whose
$j$-invariants are algebraic units. 
This connection reinforces the heuristics that 
CM points behave similarly to integral points on a curve in the context
of Siegel's Theorem. Indeed, the jacobian variety of a curve of genus $2$
defined over $\IQbar$ and with good reduction everywhere corresponds to
 an algebraic point on $\mathcal{A}_2$ that is integral with respect to 
the divisor given by products of elliptic curves
with their product polarisation. 
 Theorem \ref{thm:finiteness} is a finiteness result on the set of
 certain CM points of $\mathcal{A}_2$ that are
integral with respect to the said divisor. It would be interesting to
know if \textit{e.g.} Vojta's Theorem 0.4 on integral points on
semi-abelian varieties \cite{Vojta:semiabII} has an
analog for
$\mathcal{A}_g$ and other Shimura varieties.

Finally, one can ask if the questions posed above remain valid in an
$S$-integer setting. In other words, are there only finitely many
curves $C$ of genus $2$ or $3$ which have
good reduction above  the complement of a finite set of primes, where
$\jac{C}$ has CM,
and where possibly further conditions are met?

The paper is structured as follows. In the next section we introduce
some basic notation. In Section \ref{sec:abvar} we cover some properties of abelian varieties with
complex multiplication, and recall Shimura's Theorem 
on the Galois orbit for the cases we are interested in. In
Section \ref{sec:faltingsheight} we recall first the Faltings height
of an abelian variety. 
Then in  Section \ref{sec:colmezconj} we use a known case of Colmez's Conjecture to
express the Faltings height of certain abelian varieties with CM.  
 Section \ref{sec:hyperelliptic} contains the  local
decomposition of the Faltings height of
a jacobian surface with good reduction
at all finite places. 
The archimedean places in this decomposition are bounded from above
in Section \ref{sec:arch}.
Finally, the proof of both our theorems is completed in
Section \ref{sec:proofs}. In the appendix we both express, using Colmez's Conjecture, and approximate
numerically,
using the result in Section \ref{sec:hyperelliptic},
the Faltings height of three jacobian varieties of genus $2$
curves. Each pair of 
heights are equal up-to the prescribed precision. The computations and
statements made in the appendix are not
necessary for the proof of our theorems.

\medskip

\textbf{Acknowledgments} The authors thank Qing Liu and Show-Wu Zhang  for helpful
conversations. The second-named author is supported by ANR-10-BLAN-0115
Hamot, ANR-10-JCJC-0107 Arivaf and DNRF Niels Bohr Professorship. Both
authors thank the 
 Universit\'e de Bordeaux, the Technical University of Darmstadt, and 
the University of
Frankfurt.  They also thank  the DFG
for supporting this collaboration through 
the project ``Heights and
unlikely intersections'' HA~6828/1-1.


\section{Notation}
\label{sec:notation}

In this paper it 
 will be convenient to take $\IQbar$ as the algebraic
closure of $\IQ$ in $\IC$ and all number fields to be subfields of
$\IQbar$.

The letter $i$ stands for an element of $\overline{\mathbb{Q}}$ such that $i^2=-1$.

We let $K^\times$ denote the multiplicative group of any field $K$. 
If $K$ is a number field, then $\Delta_K$ is its discriminant and
$\cg{K}$ is the class group of $K$. 
We use the symbol $\O{K}$ for the 
 ring of integers of $K$ and $\O{K}^\times$ is the group of units of
 $\O{K}$. If $\mathfrak A$ is a fractional ideal of $K$, then 
$[\mathfrak A]$ denotes its class in $\cg{K}$.
If $K/F$ is an extension of number fields, then $\diff{K/F}$ is its
different and $\mathfrak{d}_{K/F}$ is its relative discriminant. 
The  norm of $\mathfrak A$ is $\inorm{\mathfrak A}$, so
$\inorm{\mathfrak A}=[\O{K}:\mathfrak A]$ if $\mathfrak A$ is an ideal
of $\O{K}$.
For the  norm to $\mathfrak{A}$ relative to $K/F$,
 a fractional ideal of $F$,  we use the symbol
$\norm{\mathfrak A}{K/F}$. If $\alpha\in K$ then  
$\norm{\alpha}{K/F}\in F$
and $\tr{\alpha}{K/F}\in F$ are norm and trace, respectively, of $\alpha$ relative
to $K/F$. 

A place $\nu$ of $K$ is an absolute value on $K$ whose restriction to
$\IQ$ is the
standard absolute value on $\IQ$ or a $p$-adic absolute value for some
prime number $p$. The former places are called infinite or archimedean and we write
$ \nu\mid \infty$ whereas the latter are called finite or
non-archimedean and we write
$\nu\nmid \infty$ or $\nu\mid p$.
The set 
of finite
places is $M^{0}_K$. Any
 $\nu\in M^0_K$ corresponds to a maximal ideal of $\O{K}$ and we write
 ${\rm ord}_\nu(\mathfrak A) \in\IZ$ 
for the power with which this  ideal  appears in the
factorisation of $\mathfrak A$.
If $\alpha\in K^\times$ then 
${\rm ord}_\nu(\alpha) = {\rm ord}_\nu(\alpha\O{K})$.
 We write $K_\nu$ for the completion of $K$
with respect to $\nu$ and $d_\nu = [K_\nu:\IQ_{\nu'}]$ where $\nu'$ is the
restriction of $\nu$ to $\IQ$. 

We will often use $K$ to denote a CM-field and $F$ its totally real
subfield. 
Complex conjugation on $K$ will be denoted by
$\alpha\mapsto \overline \alpha$. 
If a CM-type $\Phi$ of $K$ is given, then we write $K^*$ 
for the associated reflex field and $\Phi^*$ for the associated reflex
CM-type. 

For the field of definition of an algebraic variety we use lower case
letters, $k$ for instance.


Let $g\ge 1$ be an integer and $\IH_g$ 
 the Siegel upper half-space, \textit{i.e.}  $g\times g$ symmetric matrices with entries in
 $\mathbb{C}$ and positive definite imaginary parts.
For brevity, $\IH=\IH_1$ denotes the upper half-plane. 
The symplectic group $\SP{2g}{\IZ}$ acts on $\IH_g$
by
\begin{equation*}
\gamma Z=(\alpha Z +\beta)(\lambda Z +\mu)^{-1}
\quad\text{if}\quad
\gamma=\left(\begin{array}{ll}
\alpha & \beta\\\lambda &\mu\end{array}\right)\in \SP{2g}{\IZ}.
\end{equation*}
We  recall $Z=(z_{lm})_{1\leq l,m\leq g}\in \IH_g$ is called Siegel
reduced and lies in Siegel's fundamental domain $\sfd_g$ 
if and only if the following properties are met.
\begin{enumerate}\label{FundamentDomain}
\item[(i)] For every $\gamma\in \SP{2g}{\IZ}$ one has
$\det\imag{\gamma Z}\leq\det\imag{Z}$
where $\imagS(\cdot)$ denotes imaginary part.
\item[(ii)] The real part is bounded by
$$\left\vert\Ree( z_{lm})\right\vert\leq \frac{1}{2}\quad\text{for
all}\quad (l,m)\in\left\{1,\ldots,g\right\}^2.$$
\item[(iii)$_a$] For all $l\in \{1,\ldots,g\}$ and all
$\xi=(\xi_1,\ldots,\xi_g)\in\IZ^g$ with 
$\mathrm{gcd}(\xi_l,\ldots,\xi_g)=1$,  
we have $\trans{\xi}\imag{ Z }\xi\geq \imag{z_{ll}}$.
\item[(iii)$_b$] For all $l\in \{1,\ldots,g-1\}$ we have $\imag{
z_{l,l+1}}\geq 0.$
\end{enumerate}

The properties (iii)$_a$ and (iii)$_b$ state that $\imag{Z}$ is
Minkowski reduced. 

We write
$\diag{\alpha_1,\ldots,\alpha_g}$ for the diagonal matrix 
with diagonal elements
$\alpha_1,\ldots,\alpha_g$ which are contained in some field. 


\section{Abelian varieties} 
\label{sec:abvar}
In the next sections we collect some statements on Hilbert modular
varieties and  abelian varieties
that we require later on. 

\subsection{Hilbert modular varieties}

\label{sec:hmv}

Theorem \ref{thm:finiteness} concerns jacobian varieties whose
endomorphism algebras contain a fixed  real quadratic number field. So
Hilbert modular surfaces  arise naturally. 
In this section we discuss some properties of a fundamental set of the
action of Hilbert modular groups on $\IH^g=\IH\times\cdots\times\IH$, the
$g$-fold product of the complex upper half-plane $\IH\subseteq\IC$.
Our main reference for this section is Chapter I of van der Geer's book \cite{vdGeer}.
However,  we will work in a  slightly modified setting and 
therefore provide some additional details.

Let $F$ be a totally real number field of degree $g$ with distinct real
embeddings $\varphi_1,\ldots,\varphi_g:F\rightarrow\IR$. Throughout
this section  
 $\mathfrak a$ is a fractional ideal of $\O{F}$. Later on we will   be
mainly interested in the case $\mathfrak a = \diff{F/\IQ}^{-1}$, the
inverse of the different of $F/\IQ$.

Let $\O{F}^{\times,+}$ be the group of totally positive units in
$\O{F}$ and
\begin{equation}
\label{eq:defGLp}
\GLp{}{\O{F}\oplus{\mathfrak a}}
= \left\{
\left(
\begin{array}{cc}
a & b \\ c & d
\end{array}
\right);\,\,
a,d\in \O{F},\, b\in \mathfrak{a}^{-1},\,
c\in \mathfrak{a},\, \text{and }
ad-bc \in \O{F}^{\times,+}\right\}.
\end{equation}
The group $\GL{2}{F}$ 
 acts on $\IP^1(F)$.
Through the $g$ embeddings $\varphi_1,\ldots,\varphi_g$ its  subgroup
$\GLp{2}{F}$ of matrices with coefficients in $F$ and totally positive determinant
 acts on $\IH^g$ by fractional linear transformations.
We are interested in the restriction of this action to
the subgroup $\GLp{}{\O{F}\oplus\mathfrak a}$. As this group's center
acts trivially on $\IH^g$ let us consider also
\begin{equation}
\label{eq:definehatGamma}
  \widehat\Gamma(\mathfrak a) = \GLp{}{\O{F}\oplus\mathfrak a} /\left\{
  \left(\begin{array}{cc}
    u &  \\ & u
  \end{array}\right);\,\, u \in \O{F}^\times \right\}.
\end{equation}
The group $ \widehat\Gamma(\mathfrak a)$ also acts on $\IP^1(F)$.

The $\widehat\Gamma(\mathfrak a)$-action on $\IP^1(F)$ consists of 
 $h=\#\cg{F} < +\infty$  orbits which represent the cusps of
$\widehat \Gamma(\mathfrak a)\backslash \IH^g$. 
For $\eta = [\alpha:\beta] \in\IP^1(F)$ with $\alpha,\beta \in F$
and $\tau=(\tau_1,\ldots,\tau_g)\in\IH^g$ we define
\begin{equation*}
  \mu(\eta,\tau) = \inorm{\alpha\O{F}+\beta \mathfrak a^{-1}}^2
  \prod_{l=1}^g
\frac{\imag{\tau_l}}{|\varphi_l(\alpha)-\varphi_l(\beta)\tau_l|^2} >
0.
\end{equation*}
The quantity $\mu(\eta,\tau)^{-1/2}$ measures the distance of the point
in $\widehat \Gamma(\mathfrak a)\backslash \IH^g$ represented by $\tau$ to the cusp represented by $\eta$. 

If $\gamma=\left(
\begin{array}{cc}
  a  & b \\ c & d
\end{array}\right)
\in\GLp{2}{F}$, then  
\begin{equation}
\label{eq:disttrans}
 \mu(\gamma\eta,\gamma \tau)=\frac{\mu(\eta,\tau)
 }{\norm{\det\gamma}{F/\IQ}^2}
\frac{\inorm{\alpha'\O{F}+\beta'\mathfrak a^{-1}}^2}{\inorm{\alpha\O{F}+\beta\mathfrak a^{-1}}^2}
\end{equation}
where 
$\alpha' = a\alpha + b\beta$ and $\beta'=c\alpha + d\beta$. 

Let us study two important special cases.  
First, if $\gamma\in\GLp{}{\O{F}\oplus\mathfrak a}$, then
$\det\gamma\in\O{F}^{\times,+}$ and the ideals appearing
on the right of (\ref{eq:disttrans}) coincide. So the 
equality simplifies to $\mu(\gamma\eta,\gamma\tau) = \mu(\eta,\tau)$. 
Second, let us suppose $\gamma\in \SL{2}{F}$ and
fix a positive integer $\lambda$ with $\lambda a,\lambda d\in \O{F}$, 
$\lambda b\in \mathfrak a^{-1}$, and $\lambda c \in \mathfrak a$. 
Then $\lambda\alpha'\O{F}+\lambda\beta'\mathfrak{a}^{-1}\subseteq
\alpha\O{F}+\beta\mathfrak a^{-1}$ and so the norm of the ideal on the
left is at least the norm of the ideal on the right. 
Equality (\ref{eq:disttrans}) implies 
  $\mu(\gamma\eta,\gamma\tau) \ge  \lambda^{-2g} \mu(\eta,\tau)$.
On applying the same argument to $\gamma^{-1}$ we find
\begin{equation}
\label{eq:distcompare}
  c^{-1} \le \frac{\mu(\gamma\eta,\gamma\tau)}{\mu(\eta,\tau)} \le c
\end{equation}
where $c>0$ depends only on $\gamma$ and not on $\eta\in\IP^1(F)$ or on $\tau\in\IH^g$. 

A \textit{fundamental set} for the action of $\widehat \Gamma(\mathfrak a)$
on $\IH^g$ is a  subset of $\IH^g$ that meets 
all $\widehat\Gamma(\mathfrak a)$-orbits. 
We do not require a fundamental set to be connected
and we do not exclude that two distinct points are in the same orbit.
In the following we will
describe a fundamental set much as van der Geer's
construction of a fundamental domain  for the action of 
$\SL{2}{\O{F}}$ on $\IH^g$ in Chapter  I.3 \cite{vdGeer}.

 First, let us fix a set of representatives
$\eta_1=[\alpha_1:\beta_1],\ldots,\eta_h=[\alpha_{h}:\beta_{h}]\in\IP^1(F)$ of the cusps.
We may assume $\alpha_1=1$ and $\beta_1=0$, \textit{i.e.}
$[\alpha_1:\beta_1]=\infty$.
Only a slight variation in the argumentation of Lemma I.2.2
\cite{vdGeer} is required to obtain
\begin{equation}
  \label{eq:cuspdistlb}
\max\{\mu(\eta_1,\tau),\ldots,\mu(\eta_h,\tau)\} \gg 1;
\end{equation}
the constants implicit in
$\ll$ and $\gg$ here and below depend only on $\mathfrak a$ and the $\alpha_m, \beta_m$. 

\begin{prop}
\label{prop:fundamentalset}
  There is a closed fundamental set $\hfd(\mathfrak a)$ for the
  action
of $\widehat\Gamma(\mathfrak a)$ on $\IH^g$ with the following
property. 
 If $\tau=(\tau_1,\ldots,\tau_g)\in \hfd(\mathfrak a)$, then 
$|\real{\tau_l}|\ll 1$ 
and
\begin{equation}
\label{eq:prodimagtaulb}
 \left(\max_{1\le m\le h} \mu(\eta_m,\tau)\right)^{-2/g}\ll  
 \imag{\tau_l} 
 \ll \left(\max_{1\le m\le h} \mu(\eta_m,\tau)\right)^{1/g}
\end{equation}
for all $1\le l\le g$.
\end{prop}
\begin{proof}
A given $\tau\in\IH^g$ is in
\begin{equation*}
  S = \left\{\tau' \in\IH^g;\,\, \mu(\eta_m,\tau') =\max\{\mu(\eta_1,\tau'),\ldots,\mu(\eta_h,\tau') \}
 \right\} \quad\text{for some $m$},
\end{equation*}
 the sphere of influence of the cusp $\eta_m$. 
We abbreviate
 $\eta=\eta_m$ and $\alpha=\alpha_m$, as well as $\beta=\beta_m$. 
Thus
\begin{equation}
\label{eq:mulb}
  \mu(\eta,\tau) \gg 1
\end{equation}
by (\ref{eq:cuspdistlb}).
Let us define
 the fractional ideal $\mathfrak b =
 \alpha\O{F}+\beta\mathfrak{a}^{-1}$ of $\O{F}$. 
Next we choose $\gamma\in\SL{2}{F}$ with 
\begin{equation*}
  \gamma^{-1} = 
\left(  \begin{array}{cc}
    \alpha & \alpha^*  \\ \beta & \beta^*
  \end{array}\right)
\end{equation*}
where $\alpha^* \in (\mathfrak a \mathfrak b)^{-1}$ and $\beta^*\in \mathfrak b^{-1}$. So
$\gamma \eta = \infty$ and we observe that an
 application of (\ref{eq:distcompare})
and (\ref{eq:mulb}) yields
\begin{equation*}
  \mu(\infty,\gamma\tau)=\mu(\gamma\eta,\gamma\tau) \gg\mu(\eta,\tau)
  \gg 1.
\end{equation*}
The left-hand side is  $\imag{\tau'_1}\cdots\imag{\tau'_g}\gg 1$ where
$\gamma\tau=(\tau'_1,\ldots,\tau'_g)$.

We observe
\begin{equation}
\label{eq:conjgroup}
\gamma\widehat\Gamma(\mathfrak a)\gamma^{-1}=
  \gamma \GLp{}{\O{F}\oplus\mathfrak a} \gamma^{-1} 
=\GLp{}{\O{F}\oplus\mathfrak a\mathfrak b^2}.
\end{equation}
and use $\GLp{}{\O{F}\oplus\mathfrak a\mathfrak b^2}$ 
to act on $\gamma\tau$.
 In fact, we will use  only  elements in the
stabiliser of $\infty$, \textit{i.e.} the subgroup of upper triangular
matrices in $\GLp{}{\O{F}\oplus\mathfrak a\mathfrak b^2}$.
As in Chapter I.3 \cite{vdGeer} 
we find  $\gamma'$ in the said group such that if
$\gamma'\gamma\tau=(\tau''_1,\ldots,\tau''_g)=\tau''$ then
\begin{equation}
  \label{eq:fdineq1}
|\real{\tau''_l}|\ll 1\quad\text{and}\quad \imag{\tau''_l} \ll
\imag{\tau''_{l'}}\quad
\text{for all}\quad  1\le l,l'\le g.
\end{equation}
We note 
$\imag{\tau''_1}\cdots\imag{\tau''_g} =
\imag{\tau'_1}\cdots\imag{\tau'_g}\gg 1$ and thus 
\begin{equation}
\label{eq:fdineq2}
 \imag{\tau''_l}\gg 1\quad\text{for all}\quad 1\le l\le g.
\end{equation}

The point $\gamma^{-1}\gamma'\gamma\tau=\gamma^{-1}\tau''$ lies in the
$\widehat\Gamma(\mathfrak a)$-orbit of $\tau$ by (\ref{eq:conjgroup}).
We define  $D$  as the set of $\tau''$ that
satisfy (\ref{eq:fdineq1}) and (\ref{eq:fdineq2}).
We take 
$\gamma^{-1}D$ as a part of the fundamental set whose
entirety $\hfd(\mathfrak a)$  is obtained by taking the union of the
sets coming from all  $h$ cusps.
Observe that $\gamma^{-1}D$ is closed in $\IH^g$, and so
$\hfd(\mathfrak a)$ is closed too. 

It remains to prove that the various bounds in the assertion hold for
 $\gamma^{-1}\tau'' \in \gamma^{-1}D$. To simplify notation we write
$\tau=\gamma^{-1}\tau''$ and recall that $\gamma\eta=\infty$ still holds.
We use the second set of inequalities in (\ref{eq:fdineq1}) to bound 
$\imag{\tau''_l}\ll (\imag{\tau''_1}\cdots\imag{\tau''_g})^{1/g} =
\mu(\infty,\tau'')^{1/g}= \mu(\gamma\eta,\tau'')^{1/g}
\ll \mu(\eta,\gamma^{-1}\tau'')^{1/g}$. 
So $\imag{\tau''_l}\ll \mu(\eta,\tau)^{1/g}$ and in particular
$\mu(\eta,\tau)\gg 1$ by (\ref{eq:fdineq2}). 
We find 
$|\tau''_l|\ll \mu(\eta,\tau)^{1/g}$ as the real part of $\tau''_l$
is bounded by (\ref{eq:fdineq1}).
Now
\begin{equation*}
 \imag{\tau_l}= \imag{\gamma_l^{-1}\tau''_l} =
  \frac{\imag{\tau''_l}}{|\beta_l\tau''_l+\beta_l^*|^2}
  \ge \frac{\imag{\tau''_l}}{(|\beta_l\tau''_l| + |\beta_l^*|)^2} \gg \frac{1}
{\mu(\eta,\tau)^{2/g}}
\end{equation*}
where the subscript $l$ in $\beta_l,\beta_l^*,$ and $\gamma_l$ indicates that
 $\varphi_l$ was applied. 
This  yields the lower bound in
(\ref{eq:prodimagtaulb}).

To deduce the upper bound we split-up into two cases. 
If $\beta_l\not=0$, then $\imag{\gamma_l^{-1}\tau_l''} \le
\imag{\tau''_l}/(|\beta_l|^2 \imag{\tau''_l}^2) \ll 1$ and in particular
$\imag{\gamma^{-1}\tau_l''} \ll \mu(\eta,\tau)^{1/g}$. So the upper
bound holds in this case. 
What if $\beta_l=0$? Then $\imag{\gamma_l^{-1}\tau_l''} =
\imag{\tau''_l}/|\beta^*_l|^2$.
Further up we have seen that $\imag{\tau''_l}\ll \mu(\eta,
\tau)^{1/g}$ and the upper bound follows from this.

To bound the real part we use
\begin{equation*}
|\real{\tau_l}|=  |\real{\gamma_l^{-1}\tau''_l}| = \frac{\left|\alpha_l\beta_l|\tau''_l|^2 +
    \alpha_l^*\beta_l^* +(\alpha_l\beta_l^*+\alpha_l^*\beta_l) \real{\tau''_l}\right|}{|\beta_l\tau''_l+\beta_l^*|^2}.
\end{equation*}
The denominator is at least $|\beta_l|^2\imag{\tau''_l}^2 \gg 1$
if $\beta_l\not=0$ and it equals $|\beta^*_l|^2\gg 1$ if $\beta_l=0$.
Using elementary estimates we conclude $|\real{\tau_l}|\ll
1$ by treating separately the cases
$|\beta_l\tau''_l|> 2|\beta_l^*|$ and $|\beta_l\tau''_l|\le 2|\beta_l^*|$. 
\end{proof}

\subsection{Abelian varieties with complex multiplication}
\label{sec:abvarcm}

In this section we recall some basic facts on a certain class of
 abelian varieties with
CM. Furthermore, we prove several estimates that will play important roles in sections to come.

Let $K$ be a CM-field with $[K:\IQ]=2g$ 
and $F$ the maximal, totally real subfield of $K$.


We suppose that
 $A$ is an abelian variety of dimension $g$ defined over $\IC$
such that there is a ring homomorphism
from an order $\cO$ of $K$ into
$\Endo{A}$ which maps $1$ to the identity map on $A$.
In addition, we suppose that $A$ is  principally
polarised.

As $[K:\IQ] = 2\dim A$,
the natural action of $K$ on the tangent space of $A$ at $0\in A(\IC)$
is equivalent to a direct sum of  
embeddings $\varphi_1,\ldots,\varphi_g:K\rightarrow \IC$ which are
distinct modulo complex conjugation. In this way, $A$ gives rise to a
CM-type
$\Phi=\{\varphi_1,\ldots,\varphi_g\}$  of $K$. 
To keep notation elementary we 
fix a basis of 
 said tangent space and identify it  with $\IC^g$ such that the action of $K$ is
given by
\begin{equation*}
\alpha (z_1,\ldots,z_g) =
(\varphi_1(\alpha) z_1,\ldots,\varphi_g(\alpha) z_g)
\end{equation*}
for $(z_1,\ldots,z_g)\in\IC^g$.

By abuse of notation we write
$\Phi(\alpha)=(\varphi_1(\alpha),\ldots,\varphi_g(\alpha))$ if
$\alpha\in K$.

The period lattice of $A$ is a discrete
subgroup $\Pi\subseteq \IC^g$ of rank $2g$. 
After scaling coordinates we may suppose
that $(1,\ldots,1)\in \Pi$.

 The set
\begin{equation*}
\mathfrak M= \{ \alpha\in K;\,\, \Phi(\alpha) \in \Pi \}
\end{equation*}
is an $\O{F}$-module since $\O{F}$ acts on the period lattice via $\Phi$. It is
finitely generated as such and it contains an order of $K$. 
Moreover, $\mathfrak{M}$ is torsion-free and $\O{F}$ is a Dedekind ring
thus $\mathfrak{M}$ is a projective $\O{F}$-module. It is of rank $2$
making it isomorphic to $\O{F}\oplus \mathfrak{a}$
where $\mathfrak{a}$ is a fractional ideal of $\O{F}$. 
Now $\mathfrak a$ is uniquely determined by its ideal class and latter
on we will show that $\mathfrak a$ lies in the class of
$\diff{F/\IQ}^{-1}$. Let us fix $\omega_1,\omega_2\in K\ssm\{0\}$ 
with 
\begin{equation}
\label{eq:MOFbasis}
 \mathfrak M = \omega_1 \O{F}+\omega_2 \mathfrak{a}. 
\end{equation}
We note 
$\omega_1\overline\omega_2-\overline\omega_1\omega_2\not=0$, where as
usual $\overline{\cdot}$ denotes complex conjugation on $K$, 
and define
\begin{equation}
\label{eq:definet0}
  t_0 = (\omega_1\overline\omega_2-\overline\omega_1\omega_2)^{-1}.
\end{equation}
Observe that if the order $\cO$ equals $\O{K}$, then $\mathfrak{M}$ is
a fractional ideal of $\O{K}$. It this case we will use the symbol
$\mathfrak{A}$ to denote $\mathfrak{M}$. 

As $A$ is principally polarised it comes with an $\IR$-bilinear form 
$E:\IC^g\times\IC^g\rightarrow \IR$ which restricts to an integral symplectic form of
determinant $1$ on $\Pi\times\Pi$. 
We note that
\begin{equation*}
H(z,w)= E(iz,w) + i E(z,w)
\end{equation*}
is a positive definite hermitian form whose imaginary part is integral
on $\Pi\times \Pi$.


Our  form $E$ satisfies the condition of Theorem 4,
Chapter II in Shimura's book  \cite{Shimura}.
So there is $t\in K$ with
$\overline t=-t$ and
 $\imag{\varphi_m(t)}>0$ for all $m$,
 such that
\begin{equation}
\label{eq:defE}
E(z,w) = \sum_{j=1}^g \varphi_j(t) (\overline{z_j} {w_j}-{z_j}
\overline{w_j})
\end{equation}
for all $z=(z_1,\ldots,z_g)$ and $w=(w_1,\ldots,w_g)$ in $\IC^g$.
Then 
\begin{equation}
\label{eq:Etrace}
E(\Phi(\alpha),\Phi(\beta)) = \tr{t\overline{\alpha}\beta}{K/\IQ}
\end{equation}
for all $\alpha,\beta\in K$.

\begin{lem} 
\label{lem:OFmodule}
Let us keep the notation from above
and also set $u = t/t_0$.
\begin{enumerate}
\item[(i)] We have $u \in F$ and
  \begin{equation*}
   E(\Phi(\mu\omega_1+\lambda\omega_2),\Phi(\mu'\omega_1+\lambda'\omega_2))
   =   \tr{u(\mu'\lambda-\mu\lambda')}{F/\IQ}
  \end{equation*}
for all $\mu,\mu,\lambda,\lambda'\in F$. 
\item[(ii)] We have $u\mathfrak a = \diff{F/\IQ}^{-1}$. 
\end{enumerate}
\end{lem}
\begin{proof}As $\overline{t_0} = -t_0$ we find $\overline{u}=u$ and thus $u\in F$. 
We find
$\tr{t\mu\mu'\omega_1\overline\omega_1}{K/\IQ} = 0$
as $\mu\mu'\omega_1\overline\omega_1\in F$ and similarly
$\tr{t\lambda\lambda'\omega_2\overline\omega_2}{K/\IQ}=0$.
Therefore by (\ref{eq:Etrace}),
  \begin{alignat*}1
   E(\Phi(\mu\omega_1+\lambda\omega_2),\Phi(\mu'\omega_1+\lambda'\omega_2))
&= \tr{t(\mu\overline\omega_1 + \lambda\overline\omega_2)
(\mu'\omega_1 + \lambda'\omega_2)}{K/\IQ}\\
&=
\tr{t(\lambda\mu'\omega_1\overline\omega_2+\mu\lambda'\overline\omega_1\omega_2
  )}{K/\IQ}\\
&=\sum_{j=1}^g \varphi_j(t(
\lambda\mu'\omega_1\overline\omega_2 + 
\mu\lambda'\overline \omega_1\omega_2 
-\lambda\mu'\overline \omega_1\omega_2 
-\mu\lambda' \omega_1\overline\omega_2))\\
& = \tr{u(\lambda\mu'-\mu\lambda')}{F/\IQ}
  \end{alignat*}
where the final equality used $t=u t_0$ and (\ref{eq:definet0}). 
Part (i) follows.

The symplectic form $E$ has determinant $1$ as it corresponds to a
principal polarisation of $A$. So there exist a $\IZ$-basis
$(\mu_1,\ldots,\mu_g)$ of $\O{F}$ and a $\IZ$-basis of $(\lambda_1,\ldots,\lambda_g)$
of $\mathfrak a$ such that 
 $E(\Phi(\lambda_l\omega_2),\Phi(\mu_m\omega_1))=0$ except if $l=m$ when the value is
$1$. 
Part (i) yields 
$E(\Phi(\lambda_l\omega_2),\Phi(\mu_m \omega_1)) = 
\tr{u \mu_m \lambda_l}{F/\IQ}$. 
So if we arrange the $g$ column vectors $\Phi(\mu_m)$ to a square
matrix $U$
and do the same with $\Phi(\lambda_l)$ to obtain $\Lambda$, then 
$  \trans{U} \diag{\varphi_1(u),\ldots,\varphi_g(u)} 
\Lambda$ is the $g\times g$ unit matrix. Thus
$\det(U)\norm{u}{F/\IQ} \det(\Lambda) =1$. 
Now $|\!\det\Lambda| = \inorm{\mathfrak a}|\Delta_F|^{1/2}$
and $|\!\det U| = |\Delta_F|^{1/2}$ and thus
$|\norm{u}{F/\IQ}|\inorm{\mathfrak{a}} = |\Delta_F|^{-1}$. 
We conclude $\inorm{ u \mathfrak{a}} = \inorm{\diff{F/\IQ}^{-1}}$. 

If $\lambda\in\mathfrak a$ is arbitrary, then 
$\tr{u\lambda}{F/\IQ} = E(\Phi(\lambda\omega_2,\omega_1))$ by part
(i). This is an integer and so $u\mathfrak{a}\subset
\diff{F/\IQ}^{-1}$. 
But we proved above that these two fractional ideals have equal norm,
thus part (ii) follows.
\end{proof}

Part (ii) of the lemma above establishes our claim that $\mathfrak a$
and $\diff{F/\IQ}^{-1}$ are in the same ideal class. 
So we can take $\mathfrak{a} = \diff{F/\IQ}^{-1}$ to start out
with. Part  (ii) of the previous lemma implies $u\in\O{F}^\times$. 
We now replace $\omega_1$ and $\omega_2$ with 
$\omega_1$ and $u^{-1}\omega_2$, respectively.
With these new periods, 
\begin{equation}
\label{eq:newM}
  \mathfrak{M} = \omega_1\O{F} + \omega_2\diff{F/\IQ}^{-1}
\end{equation}
remains true but now
\begin{equation}
\label{eq:tom1om2}
t=(\omega_1\overline{\omega_2}-\overline{\omega_1}\omega_2)^{-1}.
\end{equation}
Moreover, the formula in Lemma \ref{lem:OFmodule}(i) simplifies to 
\begin{equation}
\label{eq:newE}
     E(\Phi(\mu\omega_1+\lambda\omega_2),\Phi(\mu'\omega_1+\lambda'\omega_2))
   =   \tr{\mu'\lambda-\mu\lambda'}{F/\IQ}.
\end{equation}

Next, let us consider $\tau=\omega_2/\omega_1$. 
We compute
$\varphi_l(t)^{-1} = |\varphi_l(\omega_1)|^2 (\overline {\varphi_l(\tau)} -\varphi_l(\tau))
= -2i|\varphi_l(\omega_1)|^2 \imag{\varphi_l(\tau)}$ for all $1\le l
\le g$.
Our $t$ satisfies $\real{\varphi_l(t)}=0$ and
$\imag{\varphi_l(t)}>0$. 
We conclude $\imag{\varphi_l(\tau)} > 0$ 
for all $1\le l\le g$. 
In particular, $\Phi(\tau) \in \IH^g$. 

Recall that the group
$\widehat\Gamma(\diff{F/\IQ}^{-1})$, defined in (\ref{eq:definehatGamma}), acts on $\IH^g$
and that we described a fundamental set for this action 
 in  Section \ref{sec:hmv}.
In the proposition below we use this group 
to transform $\omega_2/\omega_1$ 
to the said fundamental set. 

Let $V \subseteq\O{F}^{\times,+}$ be a set of representatives of
 $\O{F}^{\times,+}/(\O{F}^\times)^2$.
Note that $V$ is finite.

\begin{prop}
\label{prop:preparehmv}
  There exist $\omega_1,\omega_2\in K^\times$ 
with (\ref{eq:newM}), $\Phi(\omega_2/\omega_1)\in \hfd(\diff{F/\IQ}^{-1})$, and
such that there is $v\in V$ with 
\begin{equation}
\label{eq:easyform}
E(\Phi(\mu\omega_1+\lambda\omega_2),\Phi(\mu'\omega_1+\lambda'\omega_2))=\tr{v(\lambda\mu'-\lambda'\mu)}{F/\IQ}
\end{equation}
for all $\mu,\mu',\lambda,\lambda'\in F$. 
\end{prop}
\begin{proof}
According to Proposition \ref{prop:fundamentalset}
there is 
\begin{equation*}
  \gamma=\left(
  \begin{array}{cc}
    a & b \\ c & d
  \end{array}\right)
  \in
\GLp{}{\O{F}\oplus{\diff{F/\IQ}^{-1}}}
\end{equation*}
with $\gamma\Phi(\tau) \in \hfd(\diff{F/\IQ}^{-1})$.
Multiplying $\gamma$ by a scalar matrix with diagonal entry
$u\in \O{F}^\times$  does not affect $\gamma\Phi(\tau)$ and replaces
$\det \gamma$ by $u^2\det\gamma$. So we may
assume that $\det\gamma\in V$. 
We set $\omega'_1=d\omega_1 + c\omega_2$ and
$\omega'_2=b\omega_1+a\omega_2$, and find, using the definition
(\ref{eq:defGLp}),  that
(\ref{eq:newM}) again remains true.
Using
 (\ref{eq:newE}) we
obtain
\begin{equation*}
  E(\Phi(\mu\omega'_1+\lambda\omega'_2,\mu'\omega'_1+\lambda'\omega'_2))
=\tr{(\det \gamma)(\mu'\lambda-\mu\lambda')}{F/\IQ} 
\end{equation*}
for all $\mu,\mu',\lambda,\lambda'\in F$. 
Part (iii) follows on replacing $\omega_{1}$ and $\omega_{2}$ by $\omega'_{1}$ and $\omega'_{2}$, respectively.
\end{proof}

Let $(\mu_1,\ldots,\mu_g)$ be any
 $\IZ$-basis of $\O{F}$. We may find 
 a $\IZ$-basis  $(\lambda_1,\ldots,\lambda_g)$
of $\diff{F/\IQ}^{-1}$ 
such that $(\Phi(\mu_1)\omega_1,\ldots,\Phi(\lambda_g)\omega_2)$ is a
symplectic basis for $E$. 
We note that the $\lambda_l$ may depend on the symplectic form $E$ and thus
on $\mathfrak{M}$ whereas the $\mu_m$ depended only on $F$.
Let us see how to retrieve the $\lambda_1,\ldots,\lambda_g$ from the other
data.
We define $\Lambda,U\in\mat{\IR}{g}$ as the square matrices with
columns  $\Phi(\lambda_1),\ldots,\Phi(\lambda_g)$ and $\Phi(\mu_1),\ldots,\Phi(\mu_g)$, respectively. 
Relation (\ref{eq:easyform}) yields
\begin{equation}
\tr{v\lambda_l\mu_m}{F/\IQ}
=\left\{
\begin{array}{ll}
1 & \text{for $l=m$}, \\
0 & \text{for $l\not=m$.}
\end{array}\right.
\end{equation}
So $\trans{\Lambda}\diag{\varphi_1(v),\ldots,\varphi_g(v)}U$ is the $g\times g$ unit matrix.
The period matrix with respect to the symplectic basis
is 
\begin{alignat}1
  \label{eq:periodmatrix}
Z&=U^{-1} \diag{\varphi_1(\tau),\ldots,\varphi_g(\tau)}\Lambda=U^{-1}\diag{\varphi_1(v\tau),\ldots,\varphi_g(v\tau)}\trans{U^{-1}}\\
\nonumber
&=\trans{\Lambda} \diag{\varphi_1(v\tau),\ldots,\varphi_g(v\tau)}\Lambda.
\end{alignat}
It is well-known that $Z$ lies in Siegel's upper half-space $\IH_g$.  

\begin{rem}
\label{rem:periodmatg2}
  Let us assume $g=2$ and $\O{F}^{\times,+} = (\O{F}^\times)^2$. 
So $F$ is a real quadratic field of discriminant $\Delta > 0$, say,
and
we may take $V$ as above Proposition \ref{prop:preparehmv} to contain only $1$.
Thus $\O{F} =\IZ +\theta\IZ$ with $\theta = (\Delta +
\sqrt{\Delta})/2$. 
The conjugate of $\theta$ over $\IQ$ is
$\theta'=(\Delta-\sqrt{\Delta})/2$ and we consider $\theta,\theta'$ as real numbers. So
\begin{equation*}
  \left(
  \begin{array}{cc}
    \theta & 1 \\ \theta' & 1
  \end{array}\right)
\end{equation*}
becomes an admissible choice for $U$ as above. 
Say $\omega_1,\omega_2$ are as
in Proposition \ref{prop:preparehmv} with $\tau_1 = \varphi_1(\omega_2/\omega_1)$ and $ \tau_2 = \varphi_2(\omega_2/\omega_1)\in
\IC$. A brief calculation using $\det U = \theta -\theta'=  \sqrt{\Delta}$
  yields the period matrix
\begin{equation*}
Z = U^{-1}\left(
\begin{array}{cc}
  \tau_1 & \\ & \tau_2
\end{array}\right)\trans{U^{-1}}
= \frac{1}{\Delta}
\left(
\begin{array}{cc}
 \tau_1+\tau_2 & - \tau_1\theta'-\tau_2\theta\\
- \tau_1\theta'-\tau_2\theta &\tau_1\theta'^2 + \tau_2\theta^2
\end{array}\right).
\end{equation*}
\end{rem}

For the remainder of this section we suppose that
$\cO = \cO_K$ and thus that $\mathfrak{A}=\mathfrak{M}$
is a fractional ideal of $\cO_K$. 

Next we will bound how close the point represented by $Z$ lies to the
boundary of the 
coarse moduli space of principally polarised abelian varieties of
dimension $g$. 
We will do the same for the point in
$\widehat\Gamma(\diff{F/\IQ}^{-1})\backslash \IH^g$ 
represented by $\tau$ from Proposition \ref{prop:preparehmv}.

We
 define the norm of any ideal
class $[\mathfrak{A}]  \in\cg{K}$ as
the least norm of an ideal representing the said class, \textit{i.e.}
\begin{equation*}
\inorm{[\mathfrak{A}]} 
=\min\{\inorm{\mathfrak B};\,\, \mathfrak B\text{ is an ideal of $\O{K}$ in } [\mathfrak{A}] \}.
\end{equation*}

Recall that $\sfd_g$ denotes Siegel's fundamental domain, see Section
\ref{sec:notation}.

\begin{lem}
\label{lem:easy}
Let $\omega_{1}$ and $\omega_{2}$ be as in Proposition \ref{prop:preparehmv}. 
Then 
\begin{equation*}
2^g \norm{\omega_1}{K/\IQ}
\prod_{l=1}^g\imag{\varphi_l(\omega_2/\omega_1)}
= \inorm{\mathfrak{A}} |\Delta_K|^{1/2} 
=|\norm{t}{K/\IQ}|^{-1/2}. 
\end{equation*}
\end{lem}
\begin{proof}
We let $U,\Lambda\in \mat{\IR}{g}$ denote matrices
as in (\ref{eq:periodmatrix}). 
Let $\Omega_j = \diag{\varphi_1(\omega_j),\ldots,\varphi_g(\omega_j)}$
for $1\le j\le 2$. 
Then the columns of 
\begin{equation*}
\left(
\begin{array}{cc}
\Omega_1 & \Omega_2 \\ \overline{\Omega_1}& \overline{\Omega_2}
\end{array}
\right)
\left(
\begin{array}{cc}
U & \\ & \Lambda
\end{array}
\right)
\end{equation*}
constitute a $\IZ$-basis of $\Phi\times\overline\Phi(\mathfrak{A})\subseteq\IC^{2g}$. The determinant
of this product has modulus
$\inorm{\mathfrak{A}}|\Delta_K|^{1/2} =  
|\det(\Omega_1\overline{\Omega_2}-\overline{\Omega_1}\Omega_2)||\!\det
U||\!\det \Lambda|$.
The first equality follows since 
$|\!\det U|= |\Delta_F|^{1/2}$ and $|\!\det \Lambda|
= \inorm{\diff{F/\IQ}^{-1}} |\Delta_F|^{1/2} = |\Delta_F|^{-1/2}$.

To prove the second equality let $(\alpha_1,\ldots,\alpha_{2g})$ be a $\IZ$-basis of $\mathfrak{A}$.
The determinant of the matrix 
$\left(E(\Phi(\alpha_l)  ,\Phi(\alpha_m))\right)_{1\le l,m\le 2g}$ equals the determinant of
the matrix with entries
\begin{alignat*}1
2 \sum_{j=1}^g  i\varphi_j(t)
\imag{\varphi_j(\overline{\alpha_l}{\alpha_m})}
&=
2\mathrm{Im}\Big(\sum_{j=1}^g  i\varphi_j(t)
\varphi_j(\overline{\alpha_l}{\alpha_m})\Big) \\
&=2\mathrm{Re}\Big(\sum_{j=1}^g  \varphi_j(t)
\varphi_j(\overline{\alpha_l}{\alpha_m})\Big)
\end{alignat*}
where we used $i \varphi_j(t)\in \IR$. We can rewrite
these entries as
$\sum_{j=1}^{2g}  \varphi_j(t)
\varphi_j(\overline{\alpha_l}{\alpha_m})$
on augmenting $\varphi_{g+j}=\overline{\varphi_j}$.
Thus we have
\begin{equation*}
\left|\det(E(\Phi(\alpha_l),\Phi(\alpha_m)))_{1\le l,m\le 2g}\right|
= \left(\prod_{j=1}^{2g} |\varphi_j(t)|\right)
|\det (\varphi_l(\alpha_m))_{1\le l,m\le 2g}|^2=
|\norm{t}{K/\IQ}|\inorm{\mathfrak{A}}^2|\Delta_K|.
\end{equation*} 
The absolute value on the left is $1$ as 
  the polarisation on $A$ is principal. 
Our claim follows after taking the square root and rearranging terms.
\end{proof}

For the next lemma
we fix representatives $\eta_m\in\IP^1(F)$
of the $\#\cg{F}$ cusps of $\widehat{\Gamma}(\diff{F/\IQ}^{-1})
\backslash \IH^g$ as in Section \ref{sec:hmv}.

\begin{lem}
\label{lem:rgoupperbound}
Let $Z$ be the period matrix  (\ref{eq:periodmatrix}), let $\omega_{1,2}$ be as in Proposition
\ref{prop:preparehmv}, and set $\tau=\omega_2/\omega_1$. 
\begin{enumerate}
\item[(i)]
There exists a constant $c=c(g)>0$ which depends only on $g$ with the
following property. If $\gamma \in \SP{2g}{\IZ}$ with $\gamma Z \in \sfd_g$,
then 
\begin{equation*}
 \tr{\imag{\gamma Z}}{} \le c
 \left(\frac{|\Delta_K|^{1/2}}{\inorm{[\mathfrak{A}^{-1}]}} \right)^{1/g}.
\end{equation*}
\item[(ii)]
There exists a constant $c>0$ which depends only on $F$ and the
$\eta_m$ such that
  \begin{equation*}
    \mu(\eta_m,\Phi(\tau)) \le
c \frac{|\Delta_K|^{1/2}}{\inorm{[\mathfrak{A}^{-1}]}}
  \end{equation*}
for all $m$.
\end{enumerate}
\end{lem}
\begin{proof}
Let $\omega\in \mathfrak{A} \ssm \{0\}$ witness the injectivity
diameter 
\begin{equation*}
\rho
= \min\left\{H(\omega',\omega')^{1/2};\,\, \omega'\in\Pi\ssm\{0\} \right\}
> 0,
\end{equation*}
of  $A$ with its polarisation,
\textit{i.e.} 
$\rho^2 = H(\Phi(\omega),\Phi(\omega))$. 
Then
\begin{equation*}
\rho^2 = E(i \Phi(\omega),\Phi(\omega))
=2\sum_{l=1}^g |\varphi_l(t)||\varphi_l(\omega)|^2
\end{equation*}
by (\ref{eq:defE}). 
The inequality between the arithmetic mean and the geometric mean implies
\begin{equation*}
\rho^2 \ge 2g 
\left(\prod_{l=1}^{n} |\varphi_l(t)||\varphi_l(\omega)|^2\right)^{1/g}
=2g \left(|\norm{t}{K/\IQ}|^{1/2} |\norm{\omega}{K/\IQ}|\right)^{1/g}.
\end{equation*}
By the second equality in  
Lemma \ref{lem:easy} we deduce
\begin{equation*}
\rho^2 \ge 2g
\left(\frac{|\norm{\omega}{K/\IQ}|}{\inorm{\mathfrak{A}}|\Delta_K|^{1/2}}\right)^{1/g}.
\end{equation*}
Since $\omega\in \mathfrak{A}$ is non-zero there is an ideal
$\mathfrak B$ of $\O{K}$ with
$\mathfrak{A} \mathfrak B=\omega\O{K}$. Thus
$\rho^2 \ge 2g (\inorm{\mathfrak
  B}/|\Delta_K|^{1/2})^{1/g}$
since $\inorm{\mathfrak{A}}\inorm{\mathfrak B} =
|\norm{\omega}{K/\IQ}|$. 
So
\begin{equation}
\label{eq:rholb}
\rho^{-2} \le (2g)^{-1} \left(
\frac{|\Delta_K|^{1/2}}{\inorm{[\mathfrak{A}^{-1}]}}\right)^{1/g}.
\end{equation}
 since $\mathfrak B$ is  in the class
$[\mathfrak{A}^{-1}]$. 

Next we write $Z_{\rm red} = \gamma Z$ with $\gamma$ as in (i). As
$Z_{\rm red}$ lies in  Siegel's fundamental domain its imaginary part is
Minkowski reduced. 
The matrix $\imag{Z_{\rm red}}^{-1}$ represents the
hermitian form $H$ with respect to the standard basis on $\IC^g$. 
If $y'_1,\ldots,y'_g$ are the diagonal elements of $\imag{Z_{\rm red}}^{-1}$,
then we find $\rho^2 \le \min\{y'_1,\ldots,y'_g\}$ on
testing with standard basis vectors. If $y_1,\ldots,y_g$ are the
diagonal elements of $\imag{Z_{\rm red}}$, then properties of Minkowski reduced
matrices imply $y_l > 0$ and $y'_l\le c/y_l$ for all $1\le l\le g$
where $c>0$ is a constant that depends only
on $g$. 
So 
\begin{equation*}
 \rho^{-2} \ge  \max\{y_1,\ldots,y_g\}/c
\ge \tr{\imag{Y}}{}/(cg). 
\end{equation*}
We combine this inequality 
with (\ref{eq:rholb}) to deduce part (i). 

For the proof of (ii) we abbreviate $\eta=\eta_m$ and fix $\alpha\in\O{F}$ and
$\beta\in\diff{F/\IQ}^{-1}$
with $\eta=[\alpha:\beta]$.
Then
\begin{equation*}
  \mu(\eta,\Phi(\tau)) =
\inorm{\alpha\O{F}+\beta\diff{F/\IQ}}^2 |\norm{\omega_1}{K/\IQ}|
\prod_{l=1}^g \frac{\imag{\varphi_l(\tau)}}{|\varphi_l(\omega_1\alpha
  - \omega_2 \beta)|^2}
\end{equation*}
and so
\begin{equation*}
  \mu(\eta,\Phi(\tau)) =
  2^{-g}\inorm{\alpha\O{F}+\beta\diff{F/\IQ}}^2
\frac{\inorm{\mathfrak{A}}|\Delta_K|^{1/2}}{|\norm{\omega_1\alpha  - \omega_2 \beta}{K/\IQ}|}
\end{equation*}
by the first equality of Lemma \ref{lem:easy}.
We observe that $\omega_1\alpha-\omega_2\beta \in\mathfrak{A}$ is non-zero. As
above $(\omega_1\alpha-\omega_2\beta)=\mathfrak{A}\mathfrak B$, for
some ideal $\mathfrak B\in [\mathfrak{A}^{-1}]$.   We
conclude
\begin{equation*}
  \mu(\eta,\Phi(\tau))=2^{-g} \inorm{\alpha\O{F}+\beta\diff{F/\IQ}}^2
\frac{|\Delta_K|^{1/2}}{\inorm{\mathfrak B}}.
\end{equation*}
With this, part (ii) follows since $\inorm{\mathfrak B}\ge
\inorm{[\mathfrak{A}^{-1}]}$ and because $\alpha$ and $\beta$ depend
only on $F$ and the $\eta_m$. 
\end{proof}

The fact that the exponent
$1/g$ in (i) is strictly less than one for the jacobian of a genus $g=2$ curve
will prove crucial later on.

The period matrix $Z$ we constructed above may not lie
 in Siegel's fundamental domain $\sfd_g\subseteq\IH_g$ defined in Section
\ref{sec:notation}. 
We rectify this in the next lemma by using Minkowski and Siegel's reduction theory.

\begin{lem}
\label{lem:Sigmaset}
Let $\tau$ be as in Proposition \ref{prop:preparehmv}. 
For given $M > 0$ there is a finite set $\Sigma\subseteq\SP{2g}{\IZ}$
such that if 
$\max_{m} \mu(\eta_m,\Phi(\tau)) \le M$
then there exists
$\gamma \in \Sigma$ with $\gamma Z\in \sfd_g$. 
\end{lem}
\begin{proof}
In this proof, all constants implicit in $\ll$ and $\gg$ depend on
$F,$ the set $V$, the matrix $U$, the choice of cusp representatives $\eta_m$, and $M$.
So $\mu(\eta_m,\Phi(\tau))\ll 1$ for all cusp representatives
$\eta_m$. Recall that $\Phi(\tau)$ lies in the fundamental set
$\hfd(\diff{F/\IQ}^{-1})$ coming from Proposition
\ref{prop:fundamentalset}.
If $\Phi(\tau)=(\tau_1,\ldots,\tau_g)$, then 
$|\!\real{\tau_l}|\ll 1$ and 
$\imag{\tau_l}\gg 1$ for all $1\le l\le g$.

There are at most finitely many possible $\Lambda$ as in
(\ref{eq:periodmatrix}). 
Let us write $z_{lm}$ for the entries of $Z$. 
The entries of $\Lambda$ are $\varphi_l(\lambda_m)$ and so 
\begin{equation*}
  z_{lm} = \sum_{j=1}^g \varphi_{j}(v\lambda_l\lambda_m) \tau_j
\quad\text{and, in particular}\quad
  z_{ll} = \sum_{j=1}^g \varphi_{j}(v\lambda_l^2) \tau_j
\end{equation*}
for some $v\in V$. 
We observe that $\varphi_j(v)>0$ as $V\subset\O{F}^{\times,+}$ and
$\varphi_{j}(\lambda_l)\in\IR\ssm\{0\}$. So
\begin{equation*}
  |\!\imag{z_{lm}}| \le \sum_{j=1}^g |\varphi_{j}(v\lambda_l\lambda_m)|
  \imag{\tau_j}
\ll \sum_{j=1}^g \varphi_{j}(v\lambda_l^2)  \imag{\tau_j}
=\imag{z_{ll}}
\end{equation*}
for all $1\le l,m\le g$. 
Taking the determinant of the imaginary part of (\ref{eq:periodmatrix}) yields
\begin{equation*}
1\ll \prod_{j=1}^g\imag{\tau_j}
=(\det\Lambda)^{-2} \det\imag{Z} \ll \det\imag{Z}.
\end{equation*}
So $\imag{Z}$ lies in the set $Q_g(t)$ from Definition 2, Chapter I.2
\cite{Kling} for all sufficiently large $t$.

On considering the real part we obtain $|\!\real{z_{lm}}|\ll 1$ from
(\ref{eq:periodmatrix}) 
and from $|\!\real{\tau_l}|\ll 1$.
Moreover, $\imag{z_{11}} \gg \imag{\tau_1} \gg 1$. 


Hence $Z$ lies in $L_g(t)$ as in Definition 2, Chapter I.3
\cite{Kling}
for all large $t$. 
The existence of the finite set $\Sigma$ now follows from
Theorem 1, \textit{ibid.}
\end{proof}

By our relation (\ref{eq:periodmatrix}).
the entries of $Z$ are 
 contained in the normal closure of $K/\IQ$. 
In particular, the entries of $Z$ are contained in a number field
whose degree over $\IQ$ is bounded by a constant depending only on
$g$. 
We use a recent result of Pila and Tsimerman
to bound the height of a reduced
period matrix. 

\begin{lem}
\label{lem:pilatsimerman}
Let us suppose that $A$ is simple. 
If $\gamma\in\SP{2g}{\IZ}$ with $\gamma Z\in \sfd_g$, then $H(\gamma
Z)\le |\Delta_K|^c$
for a constant $c=c(g)>0$ that depends only on $g$. 
\end{lem}
\begin{proof}
This follows from Pila and Tsimerman's  Theorem 3.1
\cite{PilaTsimerman} 
as the endomorphism ring of $A$ equals $\O{K}$
under the simplicity assumption on $A$. 
\end{proof}


\subsection{The Galois orbit}
We keep the notation of the previous two sections. 

Any field automorphism $\sigma:\IC\rightarrow\IC$ determines a new abelian
variety $A^\sigma$ with complex multiplication. 
Let $\aut{\IC/K^*}$ denote the group of automorphisms that restrict to
the identity on $K^*$, the reflex field of $(K,\Phi)$. 
Shimura's Theorem  18.6 \cite{Shimura} 
describes how to recover a period lattice   of $A^\sigma$ if $\sigma
\in \aut{\IC/K^*}$. We only state a special case of
Shimura's Theorem and avoid the language of id\`eles.
Indeed, by the  assumptions of this section $\mathfrak{A}$ is a
fractional ideal in $K$ and the ideal-theoretic formulation suffices. 

To this extent let $H^*$ denote the Hilbert class field of $K^*$ and 
\begin{equation*}
  \art: \cg{K^*}\rightarrow \gal{H^*/K^*}
\end{equation*}
the group isomorphism coming from class field theory. 

The reflex norm ${\rm N}_{\Phi^*} :(K^*)^\times
\rightarrow K^\times$ is 
\begin{equation*}
  {\rm N}_{\Phi^*}(a) = \prod_{\varphi \in\Phi^*} \varphi(a),
\end{equation*}
\textit{cf.} Section 8.3 \cite{Shimura} for  standard properties including the
fact that the target is indeed $K^\times$. If $\mathfrak B^*$ is a fractional ideal of
$K^*$, then 
$\prod_{\varphi\in\Phi^*}\varphi(\mathfrak B^*)$
is a fractional ideal of $K$ which we denote with 
${\rm N}_{\Phi^*}(\mathfrak B^*)$.
Observe that ${\rm N}_{\Phi^*}$ also induces a homomorphism of class
groups $\cg{K^*}\rightarrow\cg{K}$ which we also denote by
${\rm N}_{\Phi^*}$.

\begin{thm}[Shimura]
\label{thm:shimura}
  Let $A,K,\Phi,K^*,\Phi^*,\mathfrak{A},$ and $t$ be as above and as in the last section. 
Suppose
  $\sigma\in\aut{\IC/K^*}$,  we consider $A^\sigma$ as an abelian
  variety over $\IC$.
Let  $\mathfrak B^*$ be a fractional ideal of $K^*$ with 
$\art([\mathfrak B^*]) = \sigma|_{H^*}$. 
Then $A^\sigma(\IC) \cong \IC^g / \Phi(\mathfrak{A}^\sigma)$ where
$\mathfrak{A}^\sigma = {\rm N}_{\Phi^*}(\mathfrak B^*)^{-1} \mathfrak{A}$
and $t$ transforms to $\inorm{\mathfrak{B}^*}t$. 
In particular,  the set of period lattices in  the
$\aut{\IC/K^*}$-orbit are represented by
\begin{equation*}
\left\{ [\mathfrak{A}^\sigma]; \,\,\sigma \in \aut{\IC/K^*} \right\}  = 
 {\rm N}_{\Phi^*}(\cg{K^*}) [\mathfrak{A}]=
\left\{ {\rm N}_{\Phi^*}([\mathfrak B^*])^{-1} [\mathfrak{A}] ;\,\, [\mathfrak B^*] \in \cg{K^*} \right\}. 
\end{equation*}
\end{thm}
\begin{proof}
  The first statement follows from  Theorem 18.6 part (1) \cite{Shimura}. 
Observe that $\mathfrak{A}$ is a fractional ideal, so the action 
by the finite id\`eles
 factors through the maximal compact subgroup.
The second statement is a consequence of the fact that the Artin homomorphism  is bijective.
\end{proof}

If $G$ is an abelian group, then $G[2]$ denotes its subgroup of
elements that have order dividing $2$. 

We now specialise to the case we are interested in. The following
lemma is well-known. 

\begin{lem}
\label{lem:reflexlb1}
  Suppose $K/\IQ$ is cyclic of degree $4$. 
Then $(K,\Phi)$ is primitive, $A$ is a simple abelian variety,  $K^*=K$, and
\begin{equation}
\label{eq:imagereflexnorm}
  \# {\rm N}_{\Phi^*}(\cg{K^*}) \ge  \frac{\# \cg{K} }{\# \cg{K}[2] \cdot \#
    \cg{F} }.
\end{equation}
\end{lem}
\begin{proof}
In this lemma we identify  the embeddings in the CM-type $\Phi =
\{\sigma_1,\sigma_2\}$ with automorphisms of $K$.
By hypothesis $\gal{K/\IQ}\cong\IZ/4\IZ$. As
 $\sigma_2\sigma_1^{-1}$ is neither the identity nor complex
conjugation, it must generate the Galois group. 
So $(K,\Phi)$ is primitive by Proposition
  26, Chapter II \cite{Shimura}. Therefore, $A$ is simple. 

Further down in Example 8.4 \textit{loc.~cit.}, Shimura remarks $K^*=K$ and 
$\Phi^* = \{\sigma_1^{-1},\sigma_2^{-1}\}$ under the assumption  that $(K,\Phi)$
is  primitive
 and $K/\IQ$ is abelian. 

Observe that
${\rm N}_{\Phi^*}([\mathfrak B]) = 
\sigma_1^{-1}([\mathfrak B])\sigma_2^{-1}([\mathfrak B])$
if $[\mathfrak B]\in \cg{K^*}$ and recall that
$\sigma_2\sigma_1^{-1}$ 
generates $\gal{K/\IQ}$. 
To prove the final claim it suffices to
consider the case where $\sigma_1$ is the identity and $\sigma_2$
generates the Galois group. We abbreviate $\theta=\sigma_2^{-1}$.
Let $\alpha\in \cg{K}$ be arbitrary. 
As $K$ is a CM-field with totally
real subfield $F$, the class $\alpha\overline\alpha$ is represented by
an ideal generated by an ideal of $\O{F}$. Thus there are at most
$\#\cg{F}$ different possibilities for the class
$\alpha\overline\alpha$. On the other hand,
$\alpha\theta(\alpha)(\theta(\alpha)\theta^2(\alpha))^{-1}
=\alpha \theta^2(\alpha)^{-1}= \alpha \overline\alpha^{-1}$ lies in
${\rm N}_{\Phi^*}(\cg{K^*})$. 
So $\alpha^2=(\alpha\overline\alpha) (\alpha\overline\alpha^{-1})$
lies in at most $\#\cg{F}$ translates of ${\rm N}_{\Phi^*}(\cg{K^*})$. 
The bound (\ref{eq:imagereflexnorm}) follows
because $\cg{K}$ contains precisely 
$\#\cg{K}/\#\cg{K}[2]$ squares. 
\end{proof}

\begin{lem}
\label{lem:largegaloisorbit}
Let $\epsilon > 0$. There exists a constant $c=c(\epsilon,F)>0$
depending only on $\epsilon$ and
the totally real field $F$ with the following property. 
Suppose $K/\IQ$ is cyclic of degree
$4$, then
\begin{equation*}
  \#{\rm N}_{\Phi^*}(\cg{K^*}) \ge c |\Delta_K|^{1/2-\epsilon}.
\end{equation*}
\end{lem}
\begin{proof}
Zhang's Proposition 6.3(2) \cite{zhang:equicm}  implies $\#\cg{K}[2]\le c
|\Delta_K|^{\epsilon}$
where $c$ depends only on $\epsilon>0$. 

Next we bound the class number of $K$ from below using the Brauer-Siegel
Theorem. 
For any imaginary quadratic extension $K$ of $F$ that is Galois over $\IQ$ we have
$R_K \#\cg{K} \ge c |\Delta_K|^{1/2-\epsilon}$
with a possibly smaller constant $c>0$, here $R_K>0$ denotes the
regulator of $K$. 
By Proposition 4.16 \cite{Washington}, $R_K$ is at most twice the
regulator of $F$. As we allow $c$ to depend on $F$ our lemma follows
from Lemma \ref{lem:reflexlb1}
on decreasing this constant $c$ if necessary. 
\end{proof}


\section{Faltings height}
\label{sec:faltingsheight}
We begin by recalling the definition of the Faltings height of an
abelian variety. Then we apply a known case of Colmez's Conjecture to
compute the Faltings height of certain CM abelian varieties in terms
of the $L$-functions. Finally, we will give an alternative formula for
the Faltings height for an  abelian variety that has good reduction
everywhere and that is the  jacobian of a hyperelliptic curve in genus
$2$. 

\subsection{General abelian varieties}
\label{sec:gav}
Let $A$ be an abelian variety of dimension $g\ge 1$ defined over
a number field $k$.
After extending $k$ we may suppose that $A$ has semi-stable reduction
at all finite places of $k$.
Put $S=\Spec{\mathcal O}_k$, where
${\mathcal O}_k$ is the ring of integers of $k$. Let ${\mathcal A}\longrightarrow S $
be the N\'eron model  of $A$. We shall denote by
$\varepsilon:S\rightarrow \mathcal A$
  the zero section.
We write $\Omega_{{\mathcal A}/S}^g$
for the $g$-th exterior power of the sheaf of relative differentials of 
the smooth morphism $\mathcal A\rightarrow S$. This is an invertible sheaf on $\mathcal A$ and
its pull-back
 $\varepsilon^*\Omega_{\mathcal A/S}^g$
is an invertible sheaf on $\Spec{\mathcal O}_k$.



For any embedding $\sigma$ of $k$ in $\mathbb{C}$,
the base change $\mathcal{A}_\sigma = \mathcal{A}\otimes_\sigma\IC$ is an abelian
variety over $\Spec\IC$. There is  a canonical isomorphism
$$
\varepsilon^* \Omega_{\mathcal A/S}^g \otimes_{\sigma}\mathbb{C}\simeq H^0({\mathcal
A}_{\sigma},\Omega^g_{{\mathcal A}_\sigma})\;$$
as vector spaces over $\IC$.
So we can  equip the first vector space with the  $L^2$-metric $\Vert\cdot\Vert_{\sigma}$ defined by
$$\Vert\alpha\Vert_{\sigma}^2=\frac{i^{g^2}}{c_0^{g}}\int_{{\mathcal
A}_{\sigma}(\mathbb{C})}\alpha\wedge\overline{\alpha},\;$$
for a normalizing universal constant $c_0 > 0$. 

The rank one ${\mathcal O}_k$-module 
$\varepsilon^* \Omega_{\mathcal A/S}^g$, together with the hermitian norms
$\Vert\cdot\Vert_{\sigma}$ at infinity defines an hermitian line bundle 
 over $S$. 

 Recall that for any hermitian line
bundle $\overline{\omega}$ over $S$, the Arakelov degree of $\overline{\omega}$
is defined as
$$\widehat{\degr}(\overline{\omega})=\log\#\left({\omega}/{{\mathcal
O}}_k\magicsection\right)-\sum_{\sigma\colon k\rightarrow \mathbb{C}}\log\Vert
\magicsection\Vert_{\sigma}\;,$$
where $\magicsection$ is  any non zero section of $\overline\omega$. The Arakelov degree
is independent of the choice of $\magicsection$ by the product
formula.

We now give the definition of the Faltings height, see page 354
\cite{Faltings:ES}, which is sometimes also called the differential height.

\begin{defin}\label{FaltHeight}  The  stable Faltings height of $A$ is 
$$h(A):=\frac{1}{[k:\mathbb{Q}]}\widehat{\degr}(\overline\omega)
\quad\text{where}\quad \omega = \varepsilon^* \Omega_{\mathcal
  A/S}^g$$
becomes a hermitian line bundle $\overline\omega$ when equipped with the metrics mentioned above, and we fix $c_0=2\pi$.
\end{defin}

To see that it satisfies a Northcott Theorem, see for instance Faltings's  Satz~1 \cite{Faltings:ES},
page 356 and 357 or
 the second-named author's explicit estimate \cite{Paz2}. 

For a discussion on some interesting values for $c_0$, see \cite{Paz}.
Faltings uses
 $c_0=2$. In this paper, the choice will be $c_0=2\pi$, following
 Deligne and Bost. This choice removes the $\pi$ in the expression
 derived from the Chowla-Selberg formula in the CM case. 
The choice $c_0=(2\pi)^2$ leads to a non-negative height $h_{F^+}(A)$
due to a result of Bost. 
 In any case one has the easy relations 
 \begin{equation*}
 h(A)=h_{Deligne}(A)=h_{Bost}(A)=\frac{g}{2}\log2\pi + h_{Colmez}(A)=\frac{g}{2}\log\pi+h_{Faltings}(A)=-\frac{g}{2}\log2\pi+h_{F^+}(A).  
 \end{equation*}

\subsection{Colmez's Conjecture}
\label{sec:colmezconj}

Colmez's Conjecture \cite{ColmezAnnals} relates the Faltings height of
an abelian variety with complex multiplication and the logarithmic
derivative of certain $L$-functions at $s=0$. 
In the same paper Colmez  proved his conjecture for 
CM-fields that are abelian extensions
of $\IQ$ and satisfy a ramification condition above $2$. 
Obus \cite{Obus} then generalised the result by dropping the
ramification restriction.
Yang \cite{Yang10} verified the conjecture for certain abelian surfaces
whose CM-field is not Galois over $\IQ$ cases. Our work will rely only
on the case when the CM-field is a cyclic, quartic extension of the
rationals.

Let us briefly recall Colmez's Conjecture when the CM-field $K$ is an
abelian extension of $\IQ$ of degree $2g$.
 Let $\Phi=\{\varphi_1,\ldots,\varphi_g\}$ be a CM-type of
$K$. If $\varphi:K\rightarrow \IQbar$ is an embedding, Colmez sets
 \begin{equation*}
   a_{K,\varphi,\Phi}(g_0)= 
\left\{
\begin{array}{ll}
  1 & \text{if }g_0\varphi \in \Phi, \\
  0 & \text{else wise}
\end{array}\right.
 \end{equation*}
for all $g_0\in\aGal{\IQ}$
and 
$A_{K,\Phi} = \sum_{\varphi\in\Phi}a_{K,\varphi,\Phi}$.
Then $A_{K,\Phi}$ factors through $\gal{K/\IQ}$ and by abuse of
notation we sometimes consider $A_{K,\Phi}$ as a function on this
Galois group. It is a
$\IC$-linear combination of the irreducible characters of $\gal{K/\IQ}$. Moreover, the  Artin $L$-series
attached to any character that contributes to this sum 
is  holomorphic and non-zero at $s=0$. If $\chi$ is any character of
$\gal{K/\IQ}$ then $f_\chi$ denotes the conductor of $\chi$. 

In the following result we use the normalisation of the Faltings
height used in Section \ref{sec:gav}. 

\begin{thm}[Colmez, Obus]
\label{thm:colmezobus}
  Let $A$ be an abelian variety  defined over
  a number field that is a subfield of $\IC$.
We suppose that $A$ has complex multiplication by the ring of integers
of a CM-field $K$ of degree $2\dim A$ over $\IQ$. This data provides a CM-type $\Phi$ of
$K$.
Suppose in addition that $K/\IQ$ is an abelian
extension. 
 Say
$A_{K,\Phi} =  \sum_m c_m\chi_m$ where the $\chi_m$
denote the irreducible characters of $\gal{K/\IQ}$. Then
\begin{equation*}
  h(A)  =  \left(-\sum_{m} c_m 
  \left(\frac{L'(\chi_m,s)}{L(\chi_m,s)} + \frac 12 \log f_{\chi_m}\right) + \frac{g}{2}\log 2\pi\right)\Big\vert_{s=0}
\end{equation*}
where the right hand side is evaluated at $s=0$.
\end{thm}
\begin{proof}
  We refer to
Colmez's Th\'eor\`emes 0.3(ii) and III.2.9 \cite{ColmezAnnals}
from which the result follows modulo a rational multiple of $\log 2$. 
Subsequent work of Obus \cite{Obus} removed this ambiguity.
\end{proof}



Let us consider what happens for  an abelian
surface when $K/\IQ$ is cyclic. 

\begin{prop}\label{Colmez lower bound}
Suppose $K$ is a CM-field with $K/\IQ$ cyclic of degree $4$
and let $F$ be the real quadratic subfield of $K$. 
\begin{enumerate}
\item [(i)]
Let $\Phi$ be any CM-type of $K$. If $g_0\in
\gal{K/\IQ}$, then
\begin{equation*}
A_{K,\Phi}(g_0)=\left\{
\begin{array}{ll}
  2 &\text{if $g_0=1$,}\\
  0 &\text{if $g_0$ has order $2$,}\\
  1 &\text{if $g_0$ has order $4$.}\\
\end{array}\right.
\end{equation*}
\item[(ii)] 
As a function on $\gal{K/\IQ}$ we can decompose $A_{K,\Phi} = \chi_0+\frac 12 \chi$
with $\chi_0$ the trivial character, $\chi$ is induced by the
non-trivial character of $\gal{K/F}$. Moreover, 
the conductor $f_\chi$ of $\chi$ is $\Delta_K/\Delta_F$. 
\item[(iii)]
Let $A$ be a simple abelian surface with endomorphism ring
$\O{K}$. Then
\begin{equation*}
h(A) = - \frac 12 \frac{L'(0)}{L(0)} -\frac
14 \log \frac{\Delta_K}{\Delta_F} 
 = 
 \frac 12 \frac{L'(1)}{L(1)} +\frac
14 \log \frac{\Delta_K}{\Delta_F} -\log(2\pi)- \gamma_{\IQ}
\end{equation*}
where $L(s)=\zeta_K(s)/\zeta_F(s)$ is a quotient 
of the Dedekind $\zeta$-functions of $K$ and $F$, respectively
and $\gamma_\IQ=0.577215\ldots$ denotes Euler's constant. 
\item[(iv)]
Let $A$ be as in (iii). Then
\begin{equation*}
  h(A)\ge -c + \frac{\sqrt{5}}{20}\log \Delta_K
\end{equation*}
where $c$ is a   constant that depends only on $F$.
\end{enumerate}
\end{prop}
\begin{proof}
Let us write
$\gal{K/\IQ}=\{1,h,h^2,h^3\}$. Then  $h$ has order $4$
and $h^2$ is complex conjugation on $K$. 
By definition we have $a_{K,\varphi,\Phi}(1)=1$ for all $\varphi\in\Phi$. So 
$A_{K,\Phi}(1)=2$. On the other hand, no two elements of $\Phi$ are
equal modulo complex conjugation. So $A_{K,\Phi}(h^2)=0$.
Finally, $A_{K,\varphi,\Phi}(h)\in\{0,1,2\}$. If
$\Phi=\{\varphi_1,\varphi_2\}$, then 
 simultaneous equalities $h\varphi_1=\varphi_2$ and $h\varphi_2=\varphi_1$ are
 impossible. This rules out $2$. We can also rule out $0$ since
$h\varphi_1=h^2\varphi_2$ and $h\varphi_2=h^2\varphi_1$ are impossible too. Thus
$A_{K,\varphi,\Phi}(h)=1$ and by symmetry we also find
$A_{K,\varphi,\Phi}(h^3)=1$. This completes the proof of part (i).

If $\chi$ is the  character of $\gal{K/\IQ}$ induced
by the non-trivial irreducible representation of $\gal{K/F}$, then 
\begin{equation*}
  \chi(h^k)= \left\{
  \begin{array}{ll}
    2 & \text{if $k=0$,}\\
    0 & \text{if $k=1$,}\\
    -2 & \text{if $k=2$,}\\
    0 & \text{if $k=3$.}\\
  \end{array}\right.
\end{equation*}
We observe  $A_{K,\Phi}=\chi_0+\frac 12 \chi_1$ and this yields the
first part of  (ii). 

The conductor
$f_\chi$ equals $\Delta_F\inorm{\reldisc{K/F}}$ by Proposition
VII.11.7(iii) \cite{Neukirch} where $\reldisc{K/F}$ is the relative
discriminant of $K/F$.
The final statement of part (ii) follows
as $\Delta_K=\Delta_F^2 \inorm{\reldisc{K/F}}$.

To prove  (iii)
we first remark that $L(s,\chi_0)$ is the Riemann $\zeta$-function
and that $\zeta_K(s)$ factors as $\zeta_F(s) L(s,\chi)$ with
$L(\cdot,\chi)$ the Artin $L$-function attached to the character $\chi$.
The first equality in (iii) now follows  Theorem
\ref{thm:colmezobus} applied to (ii) and since
$\zeta'(0)/\zeta(0) = \log 2\pi$. The second equality follows from the  functional
equation
of the Dedekind $\zeta$-function.


If $M$ is any number field, then $\gamma_M$ denotes the constant term
in the Taylor expansion  around $s=1$ of the logarithmic derivative of the Dedekind $\zeta$-function of $M$. 
Then $\gamma_M$ is called the  Euler-Kronecker constant of $M$ and
generalises Euler's constant $\gamma_\IQ$ to number fields. 
Badzyan's Theorem 1 \cite{Badzyan} yields the lower bound
\begin{equation*}
   \gamma_M \ge - \frac 12 \left(1-\frac{\sqrt{5}}{5}\right)\log |\Delta_M|.
\end{equation*}
We have $L'(1)/L(1)=\gamma_K-\gamma_F$ and so part (iv) follows from
Badzyan's bound together with part (iii). 
\end{proof}

Colmez \cite{Colmez} also  obtained  a  lower bound for the
Faltings height related to (iv) above.

Part (i) together with Theorem \ref{thm:colmezobus} 
implies  that the Faltings height does not depend on the CM-type of
$K$. 
This was  originally observed by Yang \cite{Yang10}. 

\subsection{Models of curves of genus $2$}
\label{sec:modelsgen2}

This section explains the choice of models used in the next section to give an explicit formula for the Faltings height of abelian surfaces. We will use Weierstrass models of degree 5 and of degree 6 for our curves of genus 2. To be able to choose models of degree 5 for a curve $C$, one needs to have at least one rational point on $C$, which can be obtained through a degree 2 field extension if needed. As plane models they are singular at infinity, one will recover the curve $C$ through desingularisation.

We work with hyperelliptic equations for a
curve $C$ of genus $2$ defined over a field $k$ of characteristic $0$. 

Suppose that $C(k)$ contains a Weierstrass point of $C$.  By
Lockhart's Proposition
1.2 \cite{Lock}  there is
a monic polynomial $P\in k[x]$ of degree $5$ such that
an open, affine subset of $C$ is isomorphic to
the affine curve determined by the equation $\mathcal E: y^2=P$. 
We call $\mathcal E$ a restricted Weierstrass equation for $C$.
Lockhart defines the discriminant of $\mathcal E$ as
$2^8\disc_{5}(P) \in k^\times$.

Say $k$ is a subfield of $\IC$. As on page 740 \cite{Lock} 
we fix an ordering on the roots of $P$ and attach a rank $4$ discrete
subgroup $\Lambda$ of $\IC^2$ and  a
period matrix $Z_{\mathcal E}\in \IH_2$ to $\mathcal E$.
We write $V_{\mathcal E} > 0$ for the covolume of $\Lambda$ in $\IC^2$.

Now we define a larger class of  Weierstrass equations. 
\begin{defin}\label{Liu} 
A  Weierstrass equation $\mathcal E$
for $C/k$ is an equation
$$\mathcal E\; :\;\,  y^{2}+Q y=P,$$
that describes an open, affine subset of $C$ where
 $P,Q\in k[x]$ with $\deg P\le 6$ and $\deg Q\le 3$.
The discriminant of $\mathcal  E$ is defined as
$\Delta_{\mathcal  E} = 2^{-12}\disc_{6}(4P+Q^2)$, it is a non-zero
element of $k$.

Suppose that $k$ is the field of fractions of a discrete valuation
ring. We call $\mathcal E$  integral if $P$ and $Q$ have integral
coefficients. The minimal discriminant $\Delta^0_{\rm min}(C)$ of $C$ is the ideal of the ring of
integers generated
by a discriminant
with minimal valuation among the discriminants of the integral
equations of $C$. In Liu's terminology \cite{Liu:cond} $\Delta^0_{\rm min}(C)$ is called the 
naive minimal discriminant. 
If $k$ is a number field, then by abuse of notation we let
$\Delta^0_{\rm min}(C)$ denote the ideal of
$\O{k}$ that is minimal at each finite place of $k$.
\end{defin}


If $\mathcal E : y^2 =P(x)$ is a restricted Weierstrass equation as
in Lockhart's work, then his notation of discriminant coincides with
the one used above, \textit{i.e.} 
$2^8\disc_{5}(P) = \Delta_{\mathcal E}$ follows from basic properties of
the discriminant. 

Weierstrass equations are  unique up to the following change of
variables, see Corollary 4.33 \cite{Liu3},
\begin{equation*}
(**)\quad
x=\frac{ax'+b}{cx'+d}  \quad\text{and}\quad
y=\frac{H(x')+e y'}{(cx'+d)^{3}} 
\end{equation*}
where $\left(\begin{array}{cc}a & b \\ c & d\end{array}\right)\in{GL_2(k)}$, $e\in{k^*}$, $H\in{k[x']}$ with $\deg H\leq 3$.

\subsection{Hyperelliptic jacobians in genus $2$}
\label{sec:hyperelliptic}
In this section we state  a formula for the Faltings height
of the jacobian of a genus $2$ curve if the said jacobian has potentially
good reduction at all finite places. 
 Ueno \cite{Ueno} had a related expression  
for the Falting height, but with a
 restriction on reduction type which is incommensurable with ours.
The second-named author proved \cite{Paz} another formula
for the Faltings height of hyperelliptic curves of any genus. 

In order to formulate our result we recall the definition of
relevant theta functions
and the $10$ non-trivial theta constants 
 in dimension $2$. The latter correspond precisely to the even characteristics 
 \begin{alignat*}1
\Theta_{1}&=\left\{ \left[\begin{array}{c}
0 \\
0 \\
0 \\
0 	
\end{array}\right], \left[\begin{array}{c}
0 \\
0 \\
0 \\
1/2 	
\end{array}\right],\left[\begin{array}{c}
0 \\
0 \\
1/2 \\
0 	
\end{array}\right],\left[\begin{array}{c}
0 \\
0 \\
1/2 \\
1/2 	
\end{array}\right]\right\},
 \\
  \Theta_{2}&=\left\{\left[\begin{array}{c}
1/2 \\
0 \\
0 \\
0 	
\end{array}\right],\left[\begin{array}{c}
0 \\
1/2 \\
0 \\
0 	
\end{array}\right],\left[\begin{array}{c}
1/2 \\
1/2 \\
0 \\
0 	
\end{array}\right],\left[\begin{array}{c}
0 \\
1/2 \\
1/2 \\
0 	
\end{array}\right],\left[\begin{array}{c}
1/2 \\
0 \\
0 \\
1/2 	
\end{array}\right],\left[\begin{array}{c}
1/2 \\
1/2 \\
1/2 \\
1/2 	
\end{array}\right] \right\}.
 \end{alignat*}
 We abbreviate $\mathcal{Z}_{2}=\Theta_{1}\cup\Theta_{2}$, the union being disjoint.
Say $\trans{(a,b)}\in\mathcal{Z}_2$ with $a,b\in\frac 12 \IZ^2$. 
We denote $Q_{ab}(n)=\trans{(n+a)}Z(n+a)+2\, \trans{(n+a)}b$, we thus
get a theta function
$$\theta_{ab}(0,Z)=\sum_{n\in{\mathbb{Z}^{2}}}e^{i\pi Q_{ab}(n)}.$$

We will use the classical Siegel cusp form
\begin{equation}
\label{def:J10}
  \displaystyle{\chi_{10}(Z)=\prod_{m\in{\mathcal{Z}_2}}\theta_m(0,
    Z)^2},
\quad\text{where}\quad Z\in\IH_2,
\end{equation}
of weight $10$, \textit{cf.} the second Remark after Proposition  2, Section 9 \cite{Kling}. So
\begin{equation*}
  Z\mapsto |\chi_{10}(Z)| \det\imag{Z}^5
\end{equation*}
is an  $\SP{4}{\IZ}$-invariant, real analytic map
$\IH_2\rightarrow\IR$.

For a finite place $\nu$  of a number field we write
\begin{equation}
\label{eq:defiota}
  \iota(\nu) = \left\{
  \begin{array}{ll}
    4 & \text{if } \nu\mid 2,\\
    3 & \text{if } \nu\mid 3,\\
    1 & \text{else wise.}
  \end{array}\right.
\end{equation}

\begin{thm}\label{hyperelliptic}
Let
$C$ be a  curve of genus $2$ defined over a number
 field $k$
such that $C(k)$ contains a Weierstrass point of $C$ and such that 
$\jac{C}$ has good reduction at all finite places of $k$. Let  $J_2,J_6,J_8,J_{10}\in k$ be Igusa's invariants
    attached to $C$. 
The following properties hold true. 
  \begin{enumerate}
  \item [(i)]  
For any
embedding $\sigma:k\rightarrow\IC$
let $Z_\sigma$ be a period matrix
coming from a restricted Weierstrass model of ${C}\otimes_{\sigma}\IC$
as in Section \ref{sec:modelsgen2}, then $\chi_{10}(Z_\sigma)\not=0$.
Moreover, we have  

    \begin{tabular}{ll}
    $\displaystyle{  h(\jac{C}) = }$ & $\displaystyle{ \frac{1}{[k:\IQ]}\Biggl(\frac{1}{60}\sum_{\nu\in{M_k^{0}}}
      \frac{d_\nu}{\iota(\nu)} 
 \log \max\left\{1,\left| J_{10}^{-\iota(\nu)} J_{2\iota(\nu)}^5\right|_\nu\right\}
 }$ \\
$$& $\displaystyle{ \quad \quad \quad \quad-
\frac{1}{10} \sum_{\sigma:k\rightarrow\IC}\log\left(2^{8}\pi^{10} |\chi_{10}( Z _{\sigma})|\det(\imagS Z _{\sigma})^{5}\right)\Biggr).}$\\
    \end{tabular}
  \item[(ii)] Let $\nu$ be a finite place of $k$, then
\begin{equation*}
  {\rm ord}_\nu\Delta^0_{\rm min}(C) = \frac{1}{\iota(\nu)} \max\left\{0, - {\rm ord}_\nu (J_{10}^{-\iota}
  J_{2\iota}^{5}) \right\}. 
\end{equation*}
    \end{enumerate}    
\end{thm}

We will prove this theorem after some preliminary work. 
But first we state an immediate corollary. 

\begin{cor}
\label{cor:faltingsheight}
  Let $C,k,$ and the $Z_\sigma$ be as in Theorem \ref{hyperelliptic}, then
  \begin{equation*}
    h(\jac{C}) = \frac{1}{[k:\IQ]}\left(\frac{1}{60}  \log \inorm{
{\Delta}^0_{\rm min}(C)} - \frac{1}{10} \sum_{\sigma:k\rightarrow\IC}
\log\left(2^{8} \pi^{10} |\chi_{10}(Z_\sigma)| \det \imag{Z_\sigma}^5\right)\right).
  \end{equation*}
\end{cor}

Suppose $C$ is a curve of genus $2$ defined over a
number field $k$ and presented by a Weierstrass equation 
as in Definition \ref{Liu}.
There exists a classical basis for 
$H^{0}(C,\Omega^{1}_{C/k})$ given by
\begin{equation*}
  \omega_1 = \frac{dx}{2y+Q(x)}\quad\text{and}\quad
\omega_{2}=\frac{xdx}{2y+Q(x)}.
\end{equation*}
Consider the section $\omega_{1}\wedge \omega_{2}\in \det {H^{0}(C,\Omega^{1}_{C/k})}$. A change of variables in
the 
Weierstrass models of $C$ leaves
\begin{equation}
\label{eq:magicsection}
  \magicsection=\Delta_{\mathcal{E}}^{{2}}(\omega_{1}\wedge\omega_{2})^{\otimes
  20}
\end{equation}
invariant, \textit{cf.} Section
1.3 \cite{Liu:cond}.


We now show how to use $\magicsection$, a differential form on the
curve $C$,  to compute the Faltings height
of the jacobian $\jac{C}$. 

Suppose $p:\mathcal{C}\rightarrow S$ is a regular semi-stable model of $C$ over
$S=\Spec\mathcal{O}_{k}$.
We now prove 
\begin{equation}
\label{eq:hjacCequality}
h(\jac{C})=\frac{1}{[k:\mathbb{Q}]}\widehat{\deg} (\det p_{*}\omega_{\mathcal{C}/S}),  
\end{equation}
where $\omega_{\mathcal C/S}$ denote the relative canonical bundle and
where  the hermitian metrics on $\det p_* \omega_{\mathcal
  C/S}$ are determined by
\begin{equation}
\label{def:archmetric}
\Vert\omega_1\wedge\omega_2\Vert_\sigma^2=
\det\left(\frac{i}{2\pi} \int_{(C\otimes_{\sigma}\IC)(\IC)}
\omega_l\wedge\overline{\omega_m}\right)_{1\le l,m\le 2}
\end{equation}
where $\sigma:k\rightarrow\IC$ denotes an embedding. 
Indeed, suppose $\varepsilon$ is a section of
$\mathcal{C}\rightarrow S$ and let $\Pic^{0}_{\mathcal
  C/S}$ be the relative Picard scheme of degree $0$.
Then $\Pic^{0}_{\mathcal C/S}$ is the identity component of the
N\'eron model $\mathcal A$ of the jacobian of $C$ by
 Theorem 4, Chapter 9.5 \cite{NeronModels}. 
This is an open subscheme of $\mathcal A$ which contains
the image of $\varepsilon$. Therefore, 
$\varepsilon^* \Omega_{\Pic^0_{\mathcal C/S}/S} = 
\varepsilon^* \Omega_{\mathcal A/S}$ which allows us to replace $\mathcal A$ by
$\Pic^0_{\mathcal C/S}$ in the computations below. 
Then
$$\Lie(\mathcal A)\simeq R^{1}p_{*}\mathcal{O}_{\mathcal{C}}.$$
Moreover by Grothendieck duality (see 6.4.3 page 243 of \cite{Liu3})
we have
$$(R^{1}p_{*}{\mathcal O}_{\mathcal{C}})^{\vee}\simeq
p_{*}\omega_{\mathcal{C}/S}. $$
Then
$$\varepsilon^{*}\Omega^{1}_{\mathcal A/S}\simeq \Lie(\mathcal A)^{\vee}\simeq
p_{*}\omega_{\mathcal{C}/S}, $$
hence
$$\varepsilon^{*}\Omega^{2}_{\mathcal A/S}\simeq \det
p_{*}\omega_{\mathcal{C}/S},$$ 
which turns out to be an isometry
  by 4.15 of the second lecture of \cite{SPA}.
We conclude (\ref{eq:hjacCequality}) by taking the Arakelov degree.

\subsubsection{Archimedian places}

We use Lockhart \cite{Lock} and Mumford \cite{Mum2} as references for these places. 
If $T$ is a subset of $\{1,2,3,4,5\}$ one then defines
$\displaystyle{m_{T}=\sum_{i\in{T}}m_{i}\in{\frac{1}{2}\mathbb{Z}^{4}}}$
with
\begin{equation*}
  m_1 = \left[
    \begin{array}{c}
       1/2 \\ 0 \\ 0 \\0
    \end{array}
\right],
  m_2 = \left[
    \begin{array}{c}
       1/2 \\ 0 \\ 1/2 \\0
    \end{array}
\right],
  m_3 = \left[
    \begin{array}{c}
       0 \\ 1/2 \\ 1/2 \\0
    \end{array}
\right],
  m_4 = \left[
    \begin{array}{c}
      0 \\ 1/2 \\ 1/2 \\ 1/2
    \end{array}
\right],
  m_5 = \left[
    \begin{array}{c}
       0 \\ 0 \\ 1/2 \\1/2
    \end{array}
\right].
\end{equation*}
We follow Lockhart's definition 3.1  \cite{Lock} and set
\begin{equation}\label{phitheta}
\Theta(Z)=\prod_{\substack{T\subseteq\{1,2,3,4,5\} \\ \#T=3}}\theta_{m_{T\circ \{1,3,5\}}}(0, Z)^8.
\end{equation}
where $\circ$ denotes the symmetric difference of sets
and  $\theta_m$ are theta functions from Section \ref{sec:hyperelliptic}. 
Observe that $\Theta(Z) = \chi_{10}(Z)^4$
and so $Z\mapsto |\Theta(Z)| \det\imag{Z}^{20}$ is an 
$\SP{4}{\IZ}$-invariant function on $\IH_2$. 

 The following result is Lockhart's Proposition 3.3 \cite{Lock}. In his
 notation we have
 $r=\binom{5}{3}=10$ and $n=\binom{4}{3}=4$. 
\begin{prop}
\label{prop:lockhart}
Let $C$ be a curve of genus $2$ defined
over $\IC$
and suppose
 $\mathcal E$ is a restricted Weierstrass equation 
for $C$.
One has the uniformisation $\jac{C}(\mathbb{C})\simeq
\mathbb{C}^{2}/\Lambda_{\mathcal E}$ with the lattice
$\Lambda_{\mathcal E}\subset\IC^2$, its  period matrix $Z_{\mathcal
  E}\in\IH_2$ and
 covolume  $V(\Lambda_{\mathcal E})$, both
 as near the beginning of Section \ref{sec:modelsgen2}. Then $|\Delta_{\mathcal
   E}|V(\Lambda_{\mathcal E})^{5}$ is independant of the equation
 $\mathcal E$ and 
$$|\Delta_{\mathcal E}|V(\Lambda_{\mathcal
   E})^{5}=2^{8}\pi^{20}\Big(|\Theta(Z_{\mathcal E})|
\det\imag{Z_{\mathcal E}}^{20}\Big)^{\frac{1}{4}}. $$
\end{prop}

 Next comes the archimedian contribution of the  section $\magicsection$ from
(\ref{eq:magicsection}).

\begin{prop}\label{hyperelliptique infinie}
Let $C$ be a curve of genus $2$ defined
over a number field $k$.
Let $\sigma:k\rightarrow \IC$ be an embedding and suppose
 $\mathcal E$ is a restricted Weierstrass equation 
for $C\otimes_\sigma \IC$.
We write $\omega_1 = dx/(2y), \omega_2 = xdx/(2y),$
and 
 $\magicsection=\Delta_{\mathcal E}^{2}(\omega_{1}\wedge\omega_{2})^{\otimes 20}$.
Then $\chi_{10}(Z_{\mathcal E})\not=0$ and 
$$\log\|\magicsection\|_{\sigma}=2\log\left(2^{8}\pi^{10}
|\chi_{10}(Z_{\mathcal E})|\det\imag{Z_{\mathcal E}}^{5}\right).$$ 
\end{prop}
\begin{proof}
We use Proposition \ref{prop:lockhart} to compute
\\

\begin{tabular}{lll}

$\|\magicsection\|_{\sigma}^{2}$ & $=$ & $\displaystyle{|\Delta_{\mathcal E}|_{\sigma}^{4}(\|\omega_{1}\wedge \omega_{2}\|_{\sigma}^{2})^{20}}$\\
\\

$$ & $=$ & $\displaystyle{|\Delta_{\mathcal
    E}|_{\sigma}^{4}
 \det\left(\frac{i}{2\pi}
  \int_{(C\otimes_{\sigma}\IC)(\IC)}\omega_l\wedge\overline{\omega_m}\right)^{20}_{1\le
    l,m\le 2}}$
\\
\\
$$ & $=$ & $\displaystyle{\frac{1}{\pi^{40}}|\Delta_{\mathcal E}|_{\sigma}^{4}V(\Lambda_{\mathcal E})^{20}}$\\
\\
$$ & $=$ & $\displaystyle{\frac{1}{\pi^{40}} 2^{32}\pi^{80}
  |\Theta(Z_{\mathcal E})|\det(\imagS Z_{\mathcal E})^{20}}$\\
\\

$$ & $=$ & $\displaystyle{2^{32}\pi^{40}|\chi_{10}(Z_\mathcal E)|^4\det(\imagS
Z_{\mathcal E})^{20}},$\\
\\

\end{tabular}

\noindent
here the second equality requires the definition
(\ref{def:archmetric}),
the next one is classical, \textit{cf.}
Chapter 2.2 \cite{GriffithsHarris}, 
the fourth one is Proposition \ref{prop:lockhart},
and the final one follows from observation $\Theta=\chi_{10}^4$ made above.
Hence 
$\|\magicsection\|_{\sigma}=\displaystyle{2^{16}\pi^{20}|\chi_{10}(Z_{\mathcal E})|^{2}\det(\imagS
  Z)^{10}}$ and it follows in particular that
$\chi_{10}(Z_{\mathcal E})\not=0$. 
\end{proof}

\subsubsection{Non-archimedian places}
\label{sec:nonarchplaces}
\begin{prop}
\label{prop:nonarch}
Let
$C$ be a curve of genus $2$ defined over a number
 field $k$.
Let $\nu$ be a finite place of $k$ 
at which 
$\jac{C}$ has good reduction.
If $\eta$ is as in (\ref{eq:magicsection}), then
\begin{equation*}
  {\rm ord}_\nu(\eta) = \frac{1}{3\iota}
\max\{0,{\rm ord}_\nu (J_{10}^\iota J_{2\iota}^{-5})\}
= \frac 13 {\rm ord}_\nu \Delta^0_{\rm min}(C)
\end{equation*}
where $\iota=\iota(\nu)$ is as in (\ref{eq:defiota}) and where the
$J_2,J_6,J_8,J_{10}$ are as in Theorem \ref{hyperelliptic}.
\end{prop}
\begin{proof}  
%
%
Let $\unr{k}_\nu$
be the maximal unramified extension of $k_\nu$ inside a fixed
algebraic closure of $k_\nu$. 
This is a strictly henselian field equipped with a discrete valuation
and whose ring of integers $\cO$ has an algebraically closed 
residue field. 
Thus $\cO$ satisfies the hypothesis needed for the references below. 

Recall that $\jac{C \otimes_k \unr{k}_\nu}$ has good reduction by
hypothesis; it has in particular  semi-stable reduction. 
So the curve $C\otimes_k \unr{k}_\nu$ has semi-stable reduction by 
 Deligne and Mumford's Theorem cited in the introduction.
The minimal regular model $f:\mathcal{C}_{\rm min}\rightarrow S$ of 
$C\otimes_{k} \unr{k}_\nu$
over $S=\Spec\cO$
is semi-stable by  Theorem 10.3.34(a) 
\cite{Liu3}. The canonical model $\mathcal{C}_{st}$, obtained
via a contraction $\mathcal{C}_{\rm
  min}\rightarrow\mathcal{C}_{st}$,
 is stable by part (b) of the same theorem \textit{loc.~cit.} 
It is well-known that  exactly 7 geometric configurations can arise
for the geometric special fibre of the stable model. They are pictured in Example 10.3.6
\textit{loc.~cit.}

We infer from  a theorem of Raynaud  that 
the special fibre of $\mathcal{C}_{st}\rightarrow S$ is either smooth or a union of $2$ elliptic
curves meeting at a point, see the paragraph before Proposition 2
\cite{Liu:stables}.

Later on, we will consider these two cases separately. But first
let us fix a Weierstrass equation $\mathcal E: y^2+Qy=P$ for
$C\otimes_k \unr{k}_\nu$ 
such that
\begin{equation*}
  \omega_1 = \frac{dx}{2y + Q} \quad\text{and}\quad
  \omega_2 = \frac{xdx}{2y + Q}
\end{equation*}
constitute an $\cO$-basis of
$H^0(\mathcal{C}_{\rm min},\omega_{\mathcal{C}_{\rm min}/S})$,
its existence is guaranteed by Proposition 2(a) \cite{Liu:cond}. 
Then
\begin{equation*}
  \magicsection = \Delta_{\mathcal E}^2 (\omega_1\wedge\omega_2)^{\otimes 20}
\in 
\det H^0(\mathcal{C}_{\rm min},\omega_{\mathcal{C}_{\rm min}/S})^{\otimes 20}
\end{equation*}
by the invariance mentioned after Theorem \ref{hyperelliptic}. 
The equation $\mathcal E$ is minimal in Liu's sense, Definition 1 \textit{loc.~cit.}, 
and we use ${\rm ord}_\nu(\Delta_{\rm min})$ to denote the order of Liu's
minimal discriminant. Observe that this order is non-negative, but may be less than the
order of the  minimal discriminant
$\Delta^0_{\rm min}(C)$. 
By Proposition 3 and its corollary, both \textit{loc.~cit.}, 
we find
\begin{equation}
  \label{eq:ordetavsordelta}
{\rm ord}(\magicsection)=2{\rm ord}_\nu(\Delta_{\rm min}).
\end{equation}

First, let us suppose that the special fibre of the stable model is not
smooth. 
Then we are in case (V) of Th\'eor\`eme 1 \cite{Liu:stables}  and
by Proposition 2 \textit{loc.~cit.} 
   $\mathcal{C}_{\rm min}$ is of type
[$I_0\text{-}I_0\text{-}m$] in Namikawa and Ueno's classification \cite{NamikawaUeno}. The
value  
\begin{equation}
\label{eq:typeV}
  m = \frac{1}{12\iota} {\rm ord}_\nu(J_{10}^\iota
  J_{2\iota}^{-5})\ge 1,
\end{equation}
computed in part (v) of the proposition 
is the thickness of the singular point in the special fiber of
$\mathcal{C}_{\rm st}$, 
here we used
 that $I_2 = J_2/12$, $I_6=J_6$, and $I_8=J_8$ in the references notation.
The fibre of $\mathcal{C}_{\rm min}\rightarrow
\mathcal{C}_{\rm st}$ 
above the unique singular point is a chain of $m-1$ copies
of the projective line. 
We shall use the Artin conductor of $\mathcal{C}_{\rm min}/S$, see the
introduction \cite{Liu:cond} for a definition. By Proposition 1 \textit{loc.~cit.}
\begin{equation}
\label{eq:case1art}
  -{\rm Art}(\mathcal{C}_{\rm min}/S) = m,
\end{equation}
indeed the conductor mentioned in the reference has exponent $0$ because
$\jac{C\otimes_k k_\nu^{unr}}$ has good reduction at $\nu$. 

Saito proved in Theorem 1 \cite{Saito:Duke88} that 
$-{\rm Art}(\mathcal{C}_{\rm min}/S)$ equals the order of yet a further
discriminant  attached to 
$\mathcal{C}_{\rm min}/S$; its definition is given in \textit{loc.~cit.} and
relies on  unpublished work of Deligne. Saito attributes this equality to Deligne
in the semi-stable case which covers our application. 
The Corollary to Proposition 3 \cite{Liu:cond} and Proposition
\cite{Saito:Comp89} yield 
\begin{equation}
\label{eq:case1orddeltamin}
  {\rm ord}_\nu (\Delta_{\rm min}) =   -{\rm Art}(\mathcal{C}_{\rm
    min}/S) + m = 2m.
\end{equation}
Using this we can relate the section $\eta$ to the Igusa invariants as follows
\begin{equation}
\label{eq:ordmagic}
  {\rm ord}(\magicsection) = 2{\rm ord}_{\nu}(\Delta_{\rm min})= 4m
= \frac{1}{3\iota}{\rm ord}_\nu(J_{10}^\iota I_{2\iota}^{-5})
\end{equation}
where the first equality used 
 (\ref{eq:ordetavsordelta}) and last equality used (\ref{eq:typeV}). 
 We obtain
\begin{equation}
\label{eq:caseVord}
  {\rm ord}(\magicsection) 
= \frac{1}{3\iota}{\rm ord}_\nu(J_{10}^\iota J_{2\iota}^{-5})
= \frac{1}{3\iota}\max\{0,{\rm ord}_\nu(J_{10}^\iota J_{2\iota}^{-5})\}
\end{equation}
and hence the first equality of this proposition in the current case.

Next we apply a result of Liu to relate ${\rm ord}(\magicsection)$ to
the order of 
the minimal discriminant $\Delta^0_{\rm min}(C)$. In Liu's notation \cite{Liu:cond} we have
$c(\mathcal{C}_{\textrm{min}}) = m$. 
His Th\'eor\`eme 2 \textit{loc.~cit.} 
implies
\begin{equation*}
  {\rm ord}_{\nu}(\Delta^0_{\rm min}(C))   
=  {\rm ord}_{\nu}(\Delta_{\rm min}) + 10m.
\end{equation*}
So 
  ${\rm ord}_{\nu}(\Delta^0_{\rm min}(C)) = 12m$
by  (\ref{eq:case1orddeltamin}). The second
equality in the assertion follows from (\ref{eq:ordmagic}).

Second,  suppose that the special fibre of $\mathcal{C}_{st}$ is
smooth. Using the same reference as above we find that the Artin
conductor
of $\mathcal{C}/S$ vanishes. 
Just as near (\ref{eq:case1orddeltamin}) we find 
${\rm ord}_\nu(\Delta_{\rm min})=0$ and
thus  
${\rm ord}_\nu(\Delta^0_{\rm min}(C))=0$
as 
$0\le {\rm ord}_\nu(\Delta^0_{\rm min}(C))\le {\rm ord}_\nu(\Delta_{\rm
  min})$ holds in general by   Proposition 2(d)  \cite{Liu:cond}.
Using (\ref{eq:ordetavsordelta}) we conclude
${\rm ord}(\magicsection)=0$.
Th\'eor\`eme 1 \cite{Liu:stables}, attributed to Igusa, states
${\rm ord}_\nu( J_{10}^{-\iota}J_{2\iota}^5) \ge 0$ for
all $\iota \in \{1,2,3,4,5\}$.
In particular,
${\rm ord}_\nu( J_{10}^{\iota}J_{2\iota}^{-5}) \le 0$
and so  the proposition holds true in this case too.
%
\end{proof}

\subsubsection{Proof of Theorem \ref{hyperelliptic}}
Since there is a $k$-rational Weierstrass point by 
hypothesis, there is 
  a restricted Weierstrass equation as in Proposition
\ref{hyperelliptique infinie}, with coefficients in $k$.
Part (i) of the theorem follows by studying the local contributions to the
Faltings height. The infinite places are handled by Proposition
\ref{hyperelliptique infinie} and the finite places are dealt with 
by Proposition \ref{prop:nonarch}. 
 Observe that the Arakelov degree of
$\magicsection$ is $20$ times the desired Faltings height of $\jac{C}$. 

Part (ii) is the second equality in  
Proposition  \ref{prop:nonarch}. 


\section{Archimedean estimates}
\label{sec:arch}
\subsection{Lower bounds for the Siegel modular form of weight $10$ in degree
$2$}\label{theta functions}

The contribution of the infinite places to the Faltings height 
in Theorem \ref{hyperelliptic} involves the Siegel modular form $\chi_{10}$ of weight $10$
and degree $2$
defined in (\ref{def:J10}). 
A lower bound for the modulus of $\chi_{10}$ can be used to bound the
height from below. 

In this section, the period matrix  $Z$ lies in Siegel's
 fundamental domain $\sfd_2$ described in Section \ref{sec:notation}.

The modular form $\chi_{10}$ vanishes at those elements
\begin{equation}
\label{eq:Znotation}
Z = \left(
\begin{array}{cc}
z_1 & z_{12} \\
z_{12} & z_2
\end{array}\right)
\end{equation}
of $\sfd_2$ for which $z_{12}$ vanishes and only at those;
\textit{cf.} the proof of
 Proposition 2 in Section 9 \cite{Kling}. 
They correspond to abelian surfaces that are products of elliptic curves;
thus they
are not jacobians of genus $2$ curves.



In the following lemmas we implicitly use techniques from the second-named
author's work \cite{Paz3} and obtain some minor numerical improvements. 
We will use $a$ and $b$ to denote components of the even
characteristic $\trans{(a,b)}\in \mathcal{Z}_{2}$ from Section \ref{sec:hyperelliptic}  and abbreviate
$$T_{ab}=\Big\{\displaystyle{n\in{\mathbb{Z}^{2}}\,|\,\imagS
Q_{ab}(n)=\min_{m\in{\mathbb{Z}^{2}}}\imagS Q_{ab}(m)}\Big\}.$$

\begin{lem}\label{carac generique}
For all $n,n'\in T_{ab}$ we have 
$$e^{i\pi Q_{ab}(n)}=e^{i\pi Q_{ab}(n')}.$$ Moreover,
$T_{ab}$ is finite and 
$$\left|\theta_{ab}(0,Z) \right|\geq 2 \# T_{ab}\cdot e^{-\pi \min_{m\in{\mathbb{Z}^{2}}}\imagS Q_{ab}(m)}-\sum_{n\in{\mathbb{Z}^{2}}}e^{-\pi\imagS Q_{ab}(n)}.$$
\end{lem}
\begin{proof}
This is Lemma 4.18 \cite{Paz3}. 
\end{proof}

\begin{lem}\label{propcarac00}
If $\trans{(a,b)}\in{\Theta_{1}}$, i.e.   $a=0$, one has
$$\left|\theta_{ab}(0,Z)\right|\geq 0.44$$
 for all $Z\in{\sfd_2}$.
\end{lem}
\begin{proof}
This follows from Proposition 4.19 \cite{Paz3}. 
\end{proof}

\begin{lem}\label{propcarac1000}
If $\trans{(a,b)} \in \Theta_2$ with
$a\not=[0,0]$ and  $a\not=[1/2,1/2]$, one has
$$
\left|\theta_{ab}(0,Z)\right|  \ge
0.75 e^{-\pi\trans{a} \imagS{Z} a}.
$$
\end{lem}
\begin{proof}
This follows from Proposition 4.20 \cite{Paz3}. 
\end{proof}

The crucial case is $a=b=[1/2,1/2]$ as the corresponding theta
constant vanishes on diagonal matrices in Siegel's fundamental
domain. 

\begin{lem}\label{propcarac1100}
If $\trans{(a,b)}=[1/2,1/2,\nu/2,\nu/2]$ with
$\nu\in{\{0,1\}}$, one has
$$\left|\theta_{ab}(0,Z)\right|
\geq 1.12|1+(-1)^\nu e^{\pi
i z_{12}}|e^{-\pi(\trans{a}\imagS Z a-\imagS z_{12})}.$$
\end{lem}
\begin{proof} 
This follows from Proposition 4.22 \cite{Paz3} and from 
\begin{equation*}
  2\left(2-\left(\sum_{m\ge 0} e^{-\frac{\pi\sqrt{3}}{4}
    m(m+1)}(2m+1)\right)^2\right) \ge 1.12. 
\end{equation*}
\end{proof}

\begin{lem}\label{exp complexe}
Let $z$ be a complex number with 
 $ |\!\real{z}| \leq \pi$. Then
 $$ \displaystyle{\vert e^{i z/2}+1\vert\geq 1} \quad\text{and}\quad
\displaystyle{\vert e^{i z}-1\vert\geq  (1-e^{-1}) \min\{1, \; \vert z\vert\}.}$$
\end{lem}
\begin{proof}
The first inequality follows from $\real{e^{i z/2}}\ge 0$. 
For the second inequality we note that
$z\mapsto (e^{iz}-1)/z$ is entire and does not vanish if
$|\!\real{z}|\leq \pi$ and $z\mapsto e^{iz}-1$ does not vanish
if $|\!\real{z}|\le \pi$ and $|z|\ge 1$.
By the maximum modulus principle applied to the reciprocals we
deduce
that the minimum  of $|e^{iz}-1|/\min\{1,|z|\}$ subject to
$|\!\real{z}|\le\pi$ is attained on $|z|=1$ or $|\!\real{z}|= \pi$. In
the latter case the quotient is $|e^{-\imag{z}}+1|\ge 1$ which is
better than the claim. Let us now suppose $|z|=1$. We  assume $|t| <
1-e^{-1}$ with $t=e^{iz}-1$, this will lead to  a
contradiction and will thus complete this proof.  The logarithm $\log(1+t) = \sum_{n\ge 1} (-1)^{n+1} t^n/n$ converges
and satisfies $e^{\log(1+t)}=1+t=e^{iz}$. So $\log(1+t)=iz+2\pi ik$
for an integer $k$.
We bound the  modulus of $\log(1+t)$ from above using
the triangle inequality and obtain
\begin{equation*}
 |z+2\pi k|= |iz+2\pi ik| \le \sum_{n\ge 1} \frac{|t|^n}{n} = -\log(1-|t|)
<-\log(1-(1-e^{-1})) = 1.
\end{equation*}
This is impossible since $|z|=1$.
\end{proof}

The next proposition combines the previous lemmas.

\begin{prop}\label{minoration delta}
For any $Z\in{\sfd_{2}}$ as in (\ref{eq:Znotation}) one
has $$\vert \chi_{10}(Z) \vert \geq
c_0\min\{1, \pi\vert z_{12}\vert\}^2\;
e^{-2\pi(\Tr(\imagS{Z})-\imagS{z_{12}})}
\ge c_0\min\{1, \pi\vert z_{12}\vert\}^2\;
e^{-2\pi\Tr \imagS{Z}},$$
 with $c_0=8\cdot 10^{-5}$. 
\end{prop}
\begin{proof}
We use
Lemmas \ref{propcarac00}, \ref{propcarac1000}, and \ref{propcarac1100} 
in connection with the definition (\ref{def:J10}) to obtain
$$
\vert \chi_{10}(Z)\vert
\geq
 0.44^8 \cdot 0.75^8 \cdot 1.12^4 
|e^{i \pi  z_{12}}+1|^2|e^{i\pi  z_{12}}-1|^2e^{4\pi \imagS{z_{12}}}
\prod_{\trans{(a,b)}\in\Theta_2} e^{-2\pi \trans{a}\imagS{Z}a}.
$$ 
Observe that $|\!\real{z_{12}}|\le 1/2$. 
The first inequality in the assertion follows from this, Lemma
\ref{exp complexe} applied to $2\pi z_{12}$
and $\pi z_{12}$,
and since the product over $\Theta_2$ equals $e^{-2\pi(\Tr\imagS{Z} + \imagS{z_{12}})}$.
The second inequality follows as $Z\in \sfd_2$ entails $\imagS z_{12}\ge 0$. 
\end{proof}











\subsection{Subconvexity}
\label{sec:subconvexity}

Let $K$ be a number field.
Say $\chi:\cg{K}\rightarrow\IC^\times$ is a character of the class
group.
We may also think of $\chi$ as a Hecke character of conductor $\O{K}$.
The $L$-series attached to the character $\chi$ is 
\begin{equation*}
  L(s,\chi) = \sum_{\mathfrak  A}
\frac{\chi([\mathfrak  A])}{\inorm{\mathfrak  A}^s}
\end{equation*}
where here and below we  sum over non-zero 
ideals $\mathfrak A$ of $\O{K}$. 
It is well-known that this
 Dirichlet series determines a meromorphic function 
on $\IC$ with at most a simple pole at $s=1$ if $\chi$ is the trivial character.

The following  subconvexity estimate
follows from Michel and Venkatesh's deep Theorem 1.1 \cite{MV}. 

\begin{thm}
\label{thm:subconvexity}
Let $F$ be a totally real number field.  There exist  constants 
$c_1>0, N>0$, and $\delta\in (0,1/4)$ depending on $F$ with
the following property. If $K/F$ is an  imaginary quadratic extension and
$\chi:\cg{K}\rightarrow \IC^\times$ is a 
 character, then
\begin{equation*}
  \left|L\left(\frac 12+it,\chi\right)\right| \le c_1 (1+|t|)^{N} |\Delta_K|^{1/4-\delta}.
\end{equation*}
\end{thm}

The  following lemma involves a well-known trick in analytic number
theory, \textit{cf.} Duke,  Friedlander, and Iwaniec's work \cite{DuFrIw02} page 574.
We shift a contour integral into the critical strip and apply
the subconvexity result cited above. 

\begin{lem}
\label{lem:cgpbound}
  Let $F$ be a totally real number field and let $\delta$ be from
  Theorem \ref{thm:subconvexity}.
There is a constant
 $c_2>0$ depending only on $F$
with the following property. 
Say $K$ is a totally imaginary quadratic extension of $F$
and let $H$ be a coset of a subgroup  of 
$\cg{K}$.
If $\epsilon\in (0,1]$ 
and $x = \epsilon |\Delta_K|^{1/2}$, then
\begin{equation}
\label{eq:zetabound}
  \frac{1}{\#H} \sum_{
\atopx{\mathfrak A}{\inorm{\mathfrak A}\le x,[\mathfrak A]\in H}}
\left(\frac{x}{\inorm{\mathfrak A}}\right)^{1/2} \le  
c_2 \epsilon^{1/2} 
\max\left\{1,
\frac{|\Delta_K|^{1/2-  \delta/2}}{\#H}\right\}.
\end{equation}
\end{lem}
\begin{proof}
We fix a smooth test function $f:(0,\infty)\rightarrow [0,\infty)$ that satisfies
\begin{equation}
\label{eq:deffsubconvex}
  f(y) = \left\{
  \begin{array}{ll}
    y^{-1/2} & \mathrm{if}\; y\in (0,1],\\
    0            &\mathrm{if}\; y\ge 2. 
  \end{array}
\right. 
\end{equation}
Its Mellin transform
\begin{equation*}
  \tilde f(s) = \int_0^\infty f(y)y^{s-1}dy 
\end{equation*}
exists if $\real{s} > 1/2$ and the Mellin inversion formula holds, \textit{cf.} Proposition
9.7.7 \cite{Cohen:II}. Using in addition Theorem
9.7.5(4) \textit{loc. cit.} we see that $\tilde f$ decays rapidly;
here this means
that if $\sigma>1/2$ is fixed and $N\ge 1$ 
then $|\tilde f(\sigma+it)|(1+|t|)^{N}$
is a bounded  function in $t\in\IR$.

For a real number $x>0$ and a character 
$\chi:\cg{K}\rightarrow \IC^\times$, we define
\begin{equation}
\label{eq:Sxpsi}
  S(x,\chi) = 
\sum_{\mathfrak A} \chi([\mathfrak A]) f\left(\frac{\inorm{\mathfrak A}}{x}\right).
\end{equation}
The sum is finite since $f$ vanishes at large
arguments.
If $\sigma\in\IR$ then $\int_{(\sigma)}$ signifies the integral along
the vertical line $\real{s} = \sigma$. 
The Mellin inversion formula leads to
\begin{alignat*}1
  S(x,\chi) &=  \frac{1}{2\pi i}\sum_{\mathfrak A}
\chi([\mathfrak A]) \int_{(2)} 
\tilde f(s)  \left(\frac{x}{\inorm{\mathfrak A}}\right)^s ds 
= \frac{1}{2\pi i} \int_{(2)}
\tilde f(s) 
\left(\sum_{\mathfrak A}
\frac{\chi([\mathfrak A])}{\inorm{\mathfrak A}^s}\right) x^s ds;
\end{alignat*}
the sum and the integral commute by the Dominant Convergence Theorem.
The inner sum is the $L$-function $L(s,\chi)$, hence
\begin{equation*}
  S(x,\chi) = \frac{1}{2\pi i}\int_{(2)} \tilde f(s) L(s,\chi)x^s
  ds. 
\end{equation*}

Let $H_0$ denote the translate of $H$ containing the unit element; it
is a subgroup of $\cg{K}$.
Suppose $\chi$ is any character with $\chi|_{H_0}=1$. 
The function $|L(\sigma + it,\chi)|$ has at most polynomial growth in the imaginary
part $t$ if $\sigma \in (1/2,1)$ is  fixed.
By a contour shift and by the decay property of $\tilde f$
we arrive at
\begin{equation*}
  S(x,\chi) =\frac{1}{2\pi i}\int_{(\sigma)}
\tilde f(s) L(s,\chi)x^s ds + \xi(\chi) \tilde f(1) 
\left(\res{s=1}\zeta_K(s)\right) x
\end{equation*}
where $\xi(\chi_0)=1$ for $\chi_0$ the trivial character
 and $\xi(\chi)=0$ if $\chi\not=\chi_0$.

Here and below $c_3,c_4,c_5,c_6,$ and $c_7$ denote positive constants that depend
only on $F , f, \delta,$ and $\sigma$ but not on $K,\chi,\epsilon,$ or
$H$.

Let $h_K$ denote the class number of $K$, $R_K$ the regulator of $K$,
and $\omega_K$ the number of roots of unity in $K$. 
The residue of $\zeta_K$ at $s=1$
is positive and at most $c_3 h_K R_K / |\Delta_K|^{1/2}$
by the analytic class number formula. 
The unit groups of $K$ and $F$ have equal rank and as in the proof of
Lemma \ref{lem:largegaloisorbit} we have
 $R_K\le c_4$ where $c_4$ may depend on $F$. 
Hence
\begin{equation}
\label{eq:Sxchibound}
  |S(x,\chi)|\le 
\frac{1}{2\pi}\int_{(\sigma)}
|\tilde f(s) L(s,\chi)|x^{\sigma} ds + 
c_5 \xi(\chi)\frac{h_K}{|\Delta_K|^{1/2}}x
\end{equation}
with $c_5 = c_3 c_4|\tilde f(1)|$.

Soon we will apply 
 the Phragm\'en-Lindel\"of Principle, \textit{cf.} Theorem 5.53 \cite{IK},  to bound $|L(s,\chi)|$ from above
 in terms of $|L(1/2+it,\chi)|$ and $|L(2+it,\chi)|$; here $s=\sigma +
 it$. Indeed, the bound
$|L(2+it,\chi)|\le \zeta(2)^{[K:\IQ]}$ is elementary
but to bound $|L(1/2+it,\chi)|$ we need 
 Theorem \ref{thm:subconvexity}. We abbreviate 
 \begin{equation}
\label{def:l}
 l(\sigma) = \frac 23 (2-\sigma)
 \end{equation}
whose graph linearly interpolates $l(1/2) = 1$ and $l(2)=0$.

We suppose first
that $\chi\not=\chi_0$. Then $L(\cdot,\chi)$ is an entire
function and we may apply the Phragm\'en-Lindel\"of Principle
directly. So 
\begin{equation*}
|L(\sigma + it,\chi)| \le c_1^{l(\sigma)}\zeta(2)^{[K:\IQ](1-l(\sigma))} (1+|t|)^{N l(\sigma)}
 |\Delta_K|^{(1/4-\delta)l(\sigma)} 
\end{equation*}
for all $t\in\IR$, where we may assume $c_1\ge 1$. To treat the trivial character we work with the
entire function $L(s,\chi)(s-1)$. 
As $|\sigma + it-1|\ge 1-\sigma > 0$ we obtain 
\begin{equation*}
|L(\sigma + it,\chi_0)| \le \frac{1}{1-\sigma} c_1^{l(\sigma)}\zeta(2)^{[K:\IQ](1-l(\sigma))} (1+|t|)^{N l(\sigma)+1}
 |\Delta_K|^{(1/4-\delta)l(\sigma)}
\end{equation*}
where that additional
$1+|t|$ appears since 
$|s-1|\le 1+|\!\imag{s}|$ if $\real{s} \in \{1/2,2 \}$. 

In any case, we have
$|L(\sigma+it,\chi)|\le c_6 (1-\sigma)^{-1}(1+|t|)^{Nl(\sigma)+1}
 |\Delta_K|^{(1/4-\delta)l(\sigma)}$ with $c_6
 =c_1 \zeta(2)^{[K:\IQ]}$. Together with (\ref{eq:Sxchibound}) and the
 decay property of $\tilde f$ we obtain
 \begin{equation*}
 |S(x,\chi)| \le c_7\left(|\Delta_K|^{(1/4-\delta)l(\sigma)}x^{\sigma}
+\xi(\chi) \frac{h_K}{|\Delta_K|^{1/2}}x\right).
 \end{equation*}
We substitute $x=\epsilon|\Delta_K|^{1/2}$ to find
 \begin{alignat}1
\nonumber
 |S(x,\chi)| &\le c_7\left(|\Delta_K|^{(1/4-\delta)l(\sigma)+\sigma/2}\epsilon^\sigma
+\xi(\chi) h_K \epsilon\right) 
\\
&\le
\label{eq:Sxchilastbound}
c_7 \epsilon^{1/2}
\left(|\Delta_K|^{(1/4-\delta)l(\sigma)+\sigma/2}
+\xi(\chi) h_K \right)
\end{alignat}
where we used $\epsilon\le\epsilon^\sigma\le\epsilon^{1/2}$ as
 $\sigma\in (1/2,1)$ and $\epsilon\in (0,1]$.

We consider the mean
\begin{equation}
\label{eq:Sxdef}
  S(x) = \frac{1}{[\cg{K}:H_0]} \sum_{\chi|_{H_0}=1}
\overline{\chi(H)} S(x,\chi)
\end{equation}
over all characters $\chi$ of $\cg{K}$ that are constant on $H$. 
Since  $\chi$ takes values on the unit circle we may bound
\begin{equation*}
  |S(x)| \le \frac{[\cg{K}:H_0]-1}{[\cg{K}:H_0]}\max\{|S(x,\chi)|;\,\,\chi|_{H_0}=1 \text{ and }\chi\not=\chi_0\}
  + \frac{|S(x,\chi_0)|}{[\cg{K}:H_0]}.
\end{equation*}
We observe that $[\cg{K}:H_0] = h_K /\#H_0$. 
The bound (\ref{eq:Sxchilastbound}) yields
\begin{equation}
\label{eq:Sxbound}
|S(x)|\le
 c_7 \epsilon^{1/2} \left(|\Delta_K|^{(1/4-\delta)l(\sigma)+\sigma/2}
+\#H_0\right). 
\end{equation}

We insert
the finite sum (\ref{eq:Sxpsi}) into (\ref{eq:Sxdef}) and rearrange
the order of summation
to obtain
\begin{equation*}
  S(x) = \frac{1}{[\cg{K}:H_0]}
\sum_{\mathfrak A}\left(\sum_{\chi|_{H_0}=1} \overline{\chi(H)}\chi([\mathfrak A])
\right)f\left(\frac{\inorm{\mathfrak A}}{x}\right).
\end{equation*}
For $\chi$ from the inner sum we have $\overline{\chi(H)}\chi([\mathfrak A]) =
\chi([\mathfrak B])$ for a fractional ideal
$\mathfrak B$ with $[\mathfrak A\mathfrak{B}^{-1}]\in H$. But
$\sum_{\chi|_{H_0}=1} \chi([\mathfrak B])$ equals $[\cg{K}:H_0]$ if $[\mathfrak B]\in H_0$
 and $0$ otherwise. 
Hence
\begin{equation*}
  S(x) = \sum_{\atopx{\mathfrak A}{[\mathfrak A]\in H}} f\left(\frac{\inorm{\mathfrak A}}{x}\right)
\ge \sum_{\atopx{\mathfrak A}{\inorm{\mathfrak A}\le x,[\mathfrak A]\in H}}
f\left(\frac{\inorm{\mathfrak A}}{x}\right)
= \sum_{\atopx{\mathfrak A}{\inorm{\mathfrak A}\le x,[\mathfrak A]\in H}}
 \left(\frac{x}{\inorm{\mathfrak A}}\right)^{1/2}
\end{equation*}
as $f$ is non-negative and by  (\ref{eq:deffsubconvex}).
We divide by $\# H=\#H_0$ and use  (\ref{eq:Sxbound})  to obtain
\begin{equation*}
  \frac{1}{\# H}\sum_{\atopx{\mathfrak A}{\inorm{\mathfrak A}\le x,[\mathfrak A]\in H}}
\left(\frac{x}{\inorm{\mathfrak A}}\right)^{1/2}
\le c_7 \epsilon^{1/2}
\left(\frac{|\Delta_K|^{ (1/4-\delta) l(\sigma)+\sigma/2 }}{\# H}
+1\right).
\end{equation*}

The lemma follows as we may fix $\sigma \in (1/2,1)$ with 
$(1/4-\delta)l(\sigma)+\sigma/2 \le 1/2-\delta/2$. 
\end{proof}

Next, we state two simple consequences of the previous proposition
that we need for our main result. 
Recall that the norm $\inorm{[\mathfrak A]}$ of an ideal class in
$[\mathfrak A]\in \cg{K}$ is the smallest
norm of  a representative. 
\begin{prop}
  \label{prop:cgpbound}
  Let $F$ and $\delta$ be as in Lemma \ref{lem:cgpbound}. 
There is a constant
 $c_8>0$ depending only on $F$
with the following property. 
Say $K$ is a totally imaginary quadratic extension of $F$
and let $H$ be a coset of a subgroup  of 
$\cg{K}$, then the following two properties hold. 
  \begin{enumerate}
  \item [(i)] We have
  \begin{equation*}
  \frac{1}{\#H} \sum_{[\mathfrak A]\in H}
\left(\frac{|\Delta_K|^{1/2}}{\inorm{[\mathfrak A^{-1}]}}\right)^{1/2} \le  
c_8 \max\left\{1,
\frac{|\Delta_K|^{1/2- \delta/2}}{\#H}\right\}.
  \end{equation*}
\item[(ii)]
 Let $\epsilon \in (0,1]$, then 
  \begin{equation*}
    \frac{1}{\#H}\#\left\{[\mathfrak A]\in H;\,\, \inorm{[\mathfrak A^{-1}]}\le \epsilon
    |\Delta_K|^{1/2} \right\}
\le c_2 \epsilon^{1/2} \max\left\{1,\frac{{|\Delta_K|^{1/2 - 
    \delta/2}}}{\# H}\right\}.
  \end{equation*}
  \end{enumerate}
\end{prop}
\begin{proof}
Let $d=[K:\IQ]$. 
By a theorem of Minkowski, any ideal class of $K$ is represented by an
ideal whose norm is at most $\epsilon|\Delta_K|^{1/2}$
where $\epsilon = \frac{d!}{d^d}(\frac 4\pi)^{r_2}$ 
and $2r_2$ is the the number of non-real embeddings
$K\rightarrow\IC$. It is well-known that $\epsilon\le
1$. 
Part (i) follows from Lemma \ref{lem:cgpbound} applied to $x =\epsilon|\Delta_K|^{1/2}$
and to the coset $H^{-1}$ since $\epsilon$ depends only on $F$. 

Part (ii)
follows from Lemma \ref{lem:cgpbound} applied to $H^{-1}$ since the terms in the sum on the left of
 (\ref{eq:zetabound}) are at least $1$. 
\end{proof}


In our application, the coset $H$ will generally have more than
$|\Delta_K|^{1/2 - \delta/2}$ elements. 
In this case, the upper bound in part (i) simplifies to $c_8$, which
depends only on $F$. 
Also, if
$\epsilon$ is sufficiently small in part (ii), then find that
 only a small proportion of elements of
$H^{-1}$ will have norm less than $\epsilon|\Delta_K|^{1/2}$. 
Geometrically speaking, these ideal classes correspond to Galois
conjugates of a CM abelian variety that lie close to the cusp in the
moduli space. So  only a small proportion
of said conjugates are near the cusp.


\section{Proof of the theorems}
\label{sec:proofs}
We begin this section by proving Theorem \ref{thm:discbound}. 

Let $F$ be as in the hypothesis. 
We fix representatives $\eta_m\in\IP^1(F)$ of cusps
of $\widehat\Gamma(\diff{F/\IQ}^{-1})\backslash \IH^2$  as in
Section \ref{sec:hmv}. 
In particular, $\eta_1 = \infty$. 
We will work with a parameter
$\epsilon \in (0,1]$ that depends only on  $F$ 
and a second parameter $\kappa \in(0,1]$ 
that only depends on $F$ and $\epsilon$. We regard $\kappa$ as small
with respect to 
$\epsilon$.
We will see how to fix these parameters in due course. 

Let $C$ be as in the theorem and suppose $k\subset\IC$ is a
number field over which $C$ is defined which we will increase at will. Let $K$ be the CM-field
of $\jac{C}$. We may suppose $k\supseteq K$.

As discussed in greater detail in the introduction, the basic strategy is to let the lower bound coming from Proposition \ref{Colmez lower bound} compete with an
upper bound  of  the Faltings
height.
To estimate the Faltings height from above we need its expression 
in Corollary \ref{cor:faltingsheight}. 
Observe that this corollary is applicable as, after possibly increasing
$k$,
the classical Theorem of Serre and Tate \cite{SerTat} states that the CM abelian
variety
 $\jac{C}$ has good reduction
everywhere.
We will show that the archimedean contribution to the Faltings height
is negligible when compared to the non-archimedean contribution.
We use notation introduced in Theorem \ref{hyperelliptic} and 
Corollary \ref{cor:faltingsheight}.
Observe that $k$ satisfies the hypothesis of the theorem after passing
to a finite field extension. By part (ii) of the theorem, the
normalised norm in (\ref{eq:DiscKlb}) does not change after
passing to a further finite extension of  $k$. This settles the last statement
of Theorem \ref{thm:discbound}.

We thus decompose 
$h(\jac{C})= h^0+h_1^\infty+h_2^\infty+h_3^\infty+h_4^\infty - \frac 45 \log 2 - \log \pi$ where
\begin{equation*}
  h^0 = \frac{1}{[k:\IQ]} \sum_{\nu\in M^0_k} \frac{1}{60}
\log\inorm{\Delta^0_{{\rm min}}(C)}
\end{equation*}
is the finite part and 
\begin{equation}
\label{eq:sumarchplaces}
h_1^\infty+h_2^\infty+h_3^\infty+h_4^\infty = -\frac{1}{[k:\IQ]} \sum_{\sigma:k\rightarrow\IC} 
\frac{1}{10} \log\left(|\chi_{10}(Z_\sigma)| \det(\imagS
Z_\nu)^5 \right)
\end{equation}
and the single $h_m^\infty$ for $m\in\{1,2,3,4\}$ are determined as follows.

Shimura's Theorem \ref{thm:shimura} describes the period matrices
coming from a Galois orbit that fixes  the reflex field $K^*$. Observe that
$K^*=K$ by Lemma \ref{lem:reflexlb1} because $K/\IQ$ is cyclic. 
So (\ref{eq:sumarchplaces}) holds where each
\begin{equation}
\label{eq:decomp4}
 h_m^\infty =  -\frac{1}{10\# {\mathrm{N}}_{\Phi_m^*}(\cg{K})}
\sum_{[\mathfrak A] \in {\mathrm{N}}_{\Phi_m^*}(\cg{K}) [\mathfrak B_m]}
\log (|\chi_{10}(Z_{\mathfrak A})|\det(\imagS Z_{\mathfrak
A})^{5})
\end{equation}
corresponds to one of the four  cosets of ${\rm Aut}(\IC/K)$ in ${\rm Aut}(\IC/\IQ)$; 
here $\Phi_m$ is a CM-type of $K$ and $\mathfrak B_m\subset\O{K}$ is a
fractional ideal. 
Observe that the terms on  the right of (\ref{eq:decomp4}) do
 not depend on the choice of a
representative
$\mathfrak A\in [\mathfrak A]$. Indeed,
we already observed that $Z\mapsto |\chi_{10}(Z)|\det(\imagS Z)^5$
is $\SP{4}{\IZ}$-invariant in Section \ref{sec:hyperelliptic}.

For any $\mathfrak A$ as in the sum (\ref{eq:decomp4}),
 Proposition \ref{prop:preparehmv}
provides 
$\tau_{\mathfrak A}$ with $\Phi_m(\tau_{\mathfrak A})$
in the fundamental set $\hfd(\diff{F/\IQ}^{-1})$ from 
Section \ref{sec:hmv}. 
The period matrices $Z_{\mathfrak A}$ are as described in
(\ref{eq:periodmatrix}).

Later on we will show  that there exists $c(\epsilon,F) >0$ depending
only on $\epsilon$ and $F$ such that  
\begin{equation}
\label{eq:wtshm}
h^\infty_m \le \epsilon^{1/2}
\log\Delta_K + c(\epsilon,F) \quad\text{for each}\quad 1\le m\le 4.
\end{equation}
Our theorem follows from this inequality and from Proposition \ref{Colmez lower bound}.

 Of course, all
$m$ can be treated in a similar manner. 
So we simplify notation by abbreviating $h^\infty=h_m^\infty$ and writing 
 $H$ for ${\mathrm{N}}_{\Phi_m^*}(\cg{K})
[\mathfrak B_m]$ and $\Phi$ for the CM-type $\Phi_m$.
Observe that $H$ is a coset in the class group $\cg{K}$.

In this new notation we have
\begin{equation*}
 h^\infty =  -\frac{1}{10 \# H}
\sum_{[\mathfrak A] \in H}
\log (|\chi_{10}(Z_{{\mathfrak A},red})|\det(\imagS Z_{\mathfrak A,red})^5)
\end{equation*}
where $Z_{{\mathfrak A},red}\in \sfd_2$ is in the $\SP{4}{\IZ}$-orbit of
$Z_{\mathfrak A}$.

Below
 $c_1,c_2,\ldots,c_8$ denote positive constants that
only depend on the real quadratic field $F$.

Taking the sign in  
 $h^\infty$ into account, we would like to  bound each logarithm in $h^\infty$ from
\textit{below} using Proposition \ref{minoration delta}.
If
 $z_{\mathfrak A,12}$ is the off-diagonal entry of the Siegel
reduced matrix $Z_{\mathfrak A,red}$, then 
$z_{\mathfrak A,12}\not=0$. 
Indeed, otherwise $Z_{\mathfrak A,red}$ is diagonal. But this is
 impossible because $\jac{C}$ is not a product of elliptic curves due
 to the fact that $K/\IQ$ is cyclic, see
Corollary 11.8.2 \cite{CAV}  and Lemma \ref{lem:reflexlb1}. 
Another way to see $z_{\mathfrak A,12}\not=0$ is by noting that
 $\chi_{10}$ restricted  to $\mathcal{F}_2$ vanishes  only on diagonal matrices
and by using the proof of Proposition \ref{hyperelliptique infinie}. 
By Proposition \ref{minoration delta}  we obtain
\begin{equation*}
 h^\infty  \le c_1 + \frac{1}{10 \# H} \sum_{[\mathfrak A] \in H}
\left( \log\max\{1,|z_{\mathfrak A,12}|^{-2}\} + 2\pi \Tr(\imagS Z_{\mathfrak
  A,red})- 5 \log\det(\imagS Z_{\mathfrak A,red})\right),
\end{equation*}
Since $Z_{\mathfrak A,red}$ is Siegel reduced we have $\det(\imagS Z_{\mathfrak A,red})\ge
c_2$. 
So the average value of $-\log \det(\imagS Z_{\mathfrak A,red})$ 
is bounded from above uniformly. After possibly increasing $c_1$ we find
\begin{equation*}
 h^\infty  \le c_1 + \frac{1}{10 \# H} \sum_{[\mathfrak A] \in H}
\left( \log\max\{1,|z_{\mathfrak A,12}|^{-2}\} + 2\pi \Tr(\imagS Z_{\mathfrak
  A,red})\right).
\end{equation*}

Next we use Lemma \ref{lem:rgoupperbound}(i)  to bound
each $\Tr(\imagS Z_{\mathfrak A,red})$ from above to get
\begin{equation*}
 h^\infty  \le c_1+ c_3\frac{1}{\# H} \sum_{[\mathfrak A] \in H}
\left(\log\max\{1,|z_{\mathfrak A,12}|^{-1}\} +
\left(\frac{\Delta_K^{1/2}}{\inorm{[\mathfrak A^{-1}]}}\right)^{1/2}\right).
\end{equation*}

We continue by tackling the terms $\Delta_K^{1/2}/\inorm{[\mathfrak
A^{-1}]}$.
The trivial bound
 that follows from
$\inorm{[\mathfrak A^{-1}]} \ge 1$ is of little use here
as it leads to an upper bound for $h^\infty$ of the magnitude
$\Delta_K^{1/4}$. 
When  compared with the logarithmic lower bound coming
from Proposition \ref{Colmez lower bound} this is not good enough to conclude (\ref{eq:wtshm}). 
We need the subconvexity bound.
Proposition \ref{prop:cgpbound}(i) 
combined with the lower bound for $\#H$ from
Lemma \ref{lem:largegaloisorbit} implies 
 that the average contribution of
$(\Delta_K^{1/2}/\inorm{\mathfrak A})^{1/2}$ is bounded from
above. Thus 
\begin{equation}
\label{eq:hinfinityb1}
 h^\infty \le c_4 + c_3\frac{1}{\# H} \sum_{[\mathfrak A] \in H}
\log\max\{1,|z_{\mathfrak A,12}|^{-1}\}.
\end{equation}


Recall that $z_{\mathfrak A,12}$ is a non-zero algebraic number of
absolute logarithmic Weil height at most $H(Z_{\mathfrak A,red})$.
As $K/\IQ$ is Galois we conclude that $[\IQ(z_{\mathfrak
A,12}):\IQ]\le 4$ using the expression (\ref{eq:periodmatrix}). 
The fundamental inequality of Liouville found in 1.5.19 \cite{BG} thus implies
$|z_{\mathfrak A,12}|\ge H(Z_{\mathfrak A,red})^{-4}$.
The height of this reduced period matrix is bounded from above
polynomially in $\Delta_K$ by Lemma \ref{lem:pilatsimerman}.
Therefore, taking the logarithm yields
\begin{equation}
\label{eq:liouvillez12}
\log |z_{\mathfrak A,12}| \ge -c_5 \log\Delta_K.
\end{equation}
We use this  inequality  to bound from above the terms in
(\ref{eq:hinfinityb1})
for which $\tau_{\mathfrak A}$ is close to one of the cusps,
\textit{i.e.} $\max_{m} \mu(\eta_m,\Phi(\tau_{\mathfrak A}))  > c_6 \epsilon^{-1}$
with $c_6=c$ the constant from Lemma \ref{lem:rgoupperbound}(ii).
We have
\begin{alignat*}1
 h^\infty &\le c_4 + c_5\left(\frac{1}{\#
 H} \sum_{\max_{m} \mu(\eta_m,\Phi(\tau_{\mathfrak A}))  > c_6 \epsilon^{-1}}1\right)
\log\Delta_K 
+c_5\frac{1}{\# H} \sum_{(*)} 
\log\max\{1,|z_{\mathfrak A,12}|^{-1}\}
\end{alignat*}
where $(*)$ abbreviates the condition 
$\max_m \mu(\eta_m,\Phi(\tau))  \le c_6 \epsilon^{-1}$
here and in the sums below.      
Observe that being close to a cusp entails
 $\inorm{[\mathfrak A^{-1}]} <\epsilon \Delta_K^{1/2}$ by Lemma \ref{lem:rgoupperbound}(ii).
Part (ii) of  Proposition \ref{prop:cgpbound} tells us that not
too many
$\tau_{\mathfrak A}$ are close to a cusp. We obtain
\begin{equation*}
 h^\infty \le c_4 + c_7\epsilon^{1/2} \max\left\{1,\frac{
 \Delta_K^{1/2-\delta/2}}{\# H} \right\}
\log\Delta_K 
+c_5\frac{1}{\# H} \sum_{(*)}
\log\max\{1,|z_{\mathfrak A,12}|^{-1}\}.
\end{equation*}
We apply Lemma  \ref{lem:largegaloisorbit} again to bound
$\Delta_K^{1/2-\delta/2}/\#H$ from above. Thus
\begin{equation*}
 h^\infty \le c_4 + c_7\epsilon^{1/2}\log\Delta_K
+c_5 \frac{1}{\# H} \sum_{(*)} 
\log\max\{1,|z_{\mathfrak A,12}|^{-1}\}.
\end{equation*}

It remains to bound the sum on the right. If some $|z_{\mathfrak A,12}|$ is small, then the
corresponding conjugate of $\jac{C}$ is close to a product of elliptic
curves in the appropriate coarse moduli space. 
To measure this proximity we require the second parameter $\kappa\in (0,1]$.
We split  the upper bound for $h^\infty$ up into a subsum
where
$|z_{\mathfrak A,12}|> \kappa$ holds and one where it does not. 
The first subsum is at most $|\log\kappa|$ and so
\begin{equation*}
 h^\infty \le c_4  +
 c_7\epsilon^{1/2}\log\Delta_K 
+ c_5|\log\kappa|+\frac{c_5}{\# H} \sum_{\substack{(*) \\ |z_{\mathfrak A,12}|\le\kappa }}
(-\log |z_{\mathfrak A,12}|).
 \end{equation*}
We use (\ref{eq:liouvillez12}) again to obtain
\begin{equation}
\label{eq:hinftybd2}
 h^\infty \le c_8(1 + |\log\kappa|) +
c_8\left(
 \epsilon^{1/2}
+\frac{1}{\# H} \sum_{\substack{ (*)\\ |z_{\mathfrak A,12}|\le\kappa }}
1\right)\log\Delta_K.
 \end{equation}

To conclude we must bound the  remaining sum in (\ref{eq:hinftybd2}).
So say $[\mathfrak A]$ corresponds to one of its terms.
The property $(*)$ implies that
$\Phi(\tau_{\mathfrak A})$ is bounded away from all cusps. 
So $\Phi(\tau_{\mathfrak A})$ lies in a compact subset $\mathcal K$ of
$\IH^2$, \textit{cf.} Proposition \ref{prop:fundamentalset}
which depends only on  $\epsilon$.
Being bounded away from the cusps entails that reducing
$Z_{\mathfrak A}$ to $Z_{\mathfrak A,red}$ requires only a
 finite subset of $\SP{4}{\IZ}$. Indeed, 
we apply
Lemma \ref{lem:Sigmaset} to $M = c_6 \epsilon^{-1}$ to obtain a finite
set $\Sigma \subset \SP{4}{\IZ}$, which depends only on $c_6$ and $\epsilon$,
such that $Z_{\mathfrak A,red} = \gamma Z_{\mathfrak A}$ for some
$\gamma\in \Sigma$. 
Therefore,
\begin{equation*}
Z_{\mathfrak A} \in \bigcup_{\gamma\in\Sigma} \gamma^{-1} \mathcal A(\kappa) 
\end{equation*}
where
\begin{equation*}
\mathcal A(\kappa)=
\left\{\left(
\begin{array}{cc}
z_1 & z_{12} \\ z_{12} & z_2
\end{array}\right)
\in \sfd_2;\,\,  |z_{12}|\le \kappa
\right\}.
\end{equation*}
Each $\mathcal A(\kappa)$ is closed in
$\mat{\IC}{2}$ and $\bigcap_{\kappa >0 }\mathcal A(\kappa)$
contains only diagonal elements.  

We can reconstruct $\Phi(\tau_{\mathfrak A})$ from $Z_{\mathfrak A}$ as
follows. 
The expression  (\ref{eq:periodmatrix}) determines
 $\# \cO_{F,+}^\times / (\cO_{F}^\times)^2$ holomorphic
mappings $\IH^2\rightarrow \IH_2$. So $\Phi(\tau_{\mathfrak A})$ lies in the
 pre-image of
$ \bigcup_{\gamma\in\Sigma} \gamma^{-1} \mathcal A(\kappa)$ under one
 of them. Recall that $\Phi(\tau_{\mathfrak A})$ lies in the compact set
 $\mathcal K$.
As $\kappa\rightarrow 0$ the hyperbolic measure of the intersection of
the said pre-image and $\mathcal K$ tends to $0$. 

Galois orbits are  equidistributed by  Zhang's Corollary 3.3 \cite{zhang:equicm} and Theorem 1.2 \cite{MV}
by Michel-Venkatesh. In particular, 
\begin{equation*}
\limsup_{\Delta_K\rightarrow
 +\infty} \frac{1}{\#H} \#\left\{\tau_{\mathfrak A};\,\, 
[\mathfrak A]\in H
\text{ and }
\max_m \mu(\eta_m,\Phi(\tau_{\mathfrak A})) \le c_6\epsilon^{-1}
\text{ and }
|z_{\mathfrak A,12}|\le \kappa \right\} 
\end{equation*}
is bounded above by an expression that tends to $0$ as
$\kappa\rightarrow 0$. We fix $\kappa$ sufficiently
small in terms of $\epsilon$ such that this limes superior is at most
$\epsilon^{1/2}$. 

We can now continue bounding   (\ref{eq:hinftybd2}) from above. If $\Delta_K$ is sufficiently
large with respect to $\epsilon$, 
then the number of terms in the sum is at most $2\epsilon^{1/2}\#H$ by
the last paragraph. Therefore, 
\begin{equation*}
h^\infty \le c_8(1+|\log \kappa| + 3\epsilon^{1/2} \log\Delta_K).
\end{equation*}
If $\Delta_K$ is not large enough, we have a similar bound
with a possibly larger $c_8$. We have thus verified
the inequality (\ref{eq:wtshm})  and therefore  Theorem \ref{thm:discbound}. \qed

\begin{proof}[Proof of Theorem \ref{thm:finiteness}]
We have seen essentially the same argument  in the introduction, let us repeat
it here again.
Let $F$ and $C$ be as in the theorem. Then
we take $C$ as defined over a sufficiently large number field $k$ with
 $\Delta^0_{\mathrm{min}}(C)=\cO_k$. If $K$ is the CM-field
 of 
$\jac{C}$, then its discriminant $\Delta_K$ is bounded from above by
 constant depending only on $F$ by Theorem \ref{thm:discbound}. 
By the Theorem of Hermite-Minkowski there are only finitely many
 possibilities for $K$. As there are only finitely many abelian surfaces over
 $\IQbar$ with CM by the maximal order of $K$, this leaves at most
 finitely many possibilities for $\jac{C}$ as an abelian variety.
But each abelian variety, such as  $\jac{C}$, carries only finitely many
principal polarizations up-to equivalence; this follows from the general Narasimhan-Nori
Theorem, or from more elementary considerations as $\jac{C}$ is
simple, 
or in a direct way using the arguments in Section \ref{sec:abvarcm}.
Thus up-to $\IQbar$-isomorphism there are only finitely many possibilities for $\jac{C}$ as a
principally polarised abelian variety. 
By Torelli's Theorem this leaves only finitely many
$\IQbar$-isomorphism classes for the curve $C$.
\end{proof}

\begin{appendix}
\section{Numerical Examples}
In this section we provide some numerical examples for our expression
of the Faltings height in Theorem \ref{hyperelliptic}. We will
approximate
$|\chi_{10}(Z_\nu)|\det\imag{Z_v}^5$ numerically and compare
the resulting sum with the conclusion of Colmez's Conjecture,
Proposition \ref{Colmez lower bound}(iii).

Let $K$ be a  CM-field that is  a quartic, cyclic extension of $\IQ$ and
has maximal totally real subfield $F$. Let $A$ be an abelian surface
defined over a number field whose endomorphism ring is $\O{K}$. 

First we describe how to compute $L'(0)/L(0)$ where $L$ is as in
 Proposition
\ref{Colmez lower bound}. For this, let $f_K \ge 1$ be the finite
part of the conductor of $K/\IQ$. In other words, $f_K$ is the least
positive integer such that $K$ is a subfield of the cyclotomic field
generated by a  root of unity of order $f_K$. 
Recall that $\Delta_K > 0$, as $K/\IQ$ is a CM-field of degree $4$,
and $\Delta_F > 0$, since $F/\IQ$ is real quadratic. By Proposition
11.9 and 11.10 in Chapter VII \cite{Neukirch} we have
\begin{equation}
\label{eq:deltacond}
  \Delta_K = f_K^2 \Delta_F. 
\end{equation}

The $L$-function $L(s)=\zeta_K(s)/\zeta_F(s)$ is a product
$L(s,\chi)L(s,\overline\chi)$ of Dirichlet $L$-functions
for some character $\chi : (\IZ/f_K\IZ)^\times \rightarrow\IC$ of
order $4$. If $(\IZ/f_K\IZ)^\times$ is cyclic, \textit{e.g.} if $f_K$ is a
prime, then $\chi$ is uniquely determined up-to
complex conjugation. 

We use (\ref{eq:deltacond}) and Proposition \ref{Colmez lower
  bound}(iii) to compute
\begin{equation*}
  h(A) = - \frac 12
  \log f_K - \realS \frac{L'(0,\chi)}{L(0,\chi)}.
\end{equation*}

Observe that $\chi$ is an odd character.  Corollary 10.3.2 and
Proposition 10.3.5(1)
\cite{Cohen:II} allow us to compute $L(0,\chi)$ and $L'(0,\chi)$,
respectively. We find
\begin{equation}
\label{eq:hAalt}
  h(A) = \frac 12 \log f_K + f_K  \realS
  \left(\frac{\displaystyle{\sum_{m=1}^{f_K-1}}\chi(m)
    \log\Gamma\left(\frac{m}{f_K}\right)}
{\displaystyle{\sum_{m=1}^{f_K-1}}\chi(m)m}\right)
\end{equation}
where $\Gamma(\cdot)$ is the  gamma function. 

To compute the Igusa invariants $J_2,J_4,J_6,J_8,J_{10}$ of a
hyperelliptic equation
we use   Rodriguez-Villegas's  {\tt pari/gp} package based on work
of Mestre and Liu. We used the same software to determine the places
of potentially good reduction for the curves listed below.

We consider three curves  of genus $2$ defined over over the
rationals.
The first quite obviously has a  jacobian variety with CM. Van Wamelen
 \cite{vanWamelen,vanWamelen2}  verified this in the remaining two cases.
The source of the CM-fields $K$ in the second and third example
is van Wamelen's table \cite{vanWamelen}.
For examples 2 and 3 van Wamelen does not prove that the endomorphism ring is
the full ring of integers of $K$. But equality is compatible with our
computations below. 
We use the symbol $\dot{=}$ to denote conditional equality, subject 
to the hypothesis that the endomorphism ring of the jacobian under
consideration is indeed the full ring of integers of the CM field. 
 In all three cases, $K$ has
trivial class group. 

\bigskip
\noindent {\bf Example 1.}
We consider the curve $C$ defined by
\begin{equation*}
  y^2   = x^5 - 1.
\end{equation*}
Let $\zeta=e^{2\pi i /5}$ be a primitive $5$th root of unity. Then $(x,y)\mapsto
(\zeta x,y)$ is a automorphism of $C$ of order $5$ defined over the
cyclotomic field  $K=\IQ(\zeta)$. 
So the endomorphism ring of $\jac{C}$ over the algebraic closure
contains $\IZ[\zeta] = \O{K}$. 
The two must be equal. 
 Observe that
$F=\IQ(\sqrt{5})$ is the maximal totally real subfield of $K$ and $f_K = 5$. As a character $\chi$ near
(\ref{eq:hAalt}) we take for example
$\chi(1)=1,\chi(2)=i,\chi(3)=-i,\chi(4)=-1$. So
\begin{equation}
\label{eq:hJC1}
  h(\jac{C}) = \frac 12 \log 5 
+ \frac 12 \log\left(\Gamma\left(\frac 15\right)^{-3}
\Gamma\left(\frac 25\right)^{-1}
\Gamma\left(\frac 35\right)
\Gamma\left(\frac 45\right)^{3}\right)=
-1.4525092396456\ldots.
\end{equation}
 Bost, Mestre, Moret-Bailly \cite{BMM90} 
computed this Faltings height using a different approach
 to be
 \begin{equation*}
h(\jac{C})=2\log2\pi-\frac{1}{2}\log\left(\Gamma\Big(\frac{1}{5}\Big)^{5}\Gamma\Big(\frac{2}{5}\Big)^{3}\Gamma\Big(\frac{3}{5}\Big)\Gamma\Big(\frac{4}{5}\Big)^{-1}\right).
 \end{equation*}
This expression equals (\ref{eq:hJC1})
by classical properties of the gamma function.

The Igusa invariants of $C$ are
\begin{equation*}
(J_2,J_4,J_6,J_8,J_{10})=(0,0,0,0,2^{-12} \cdot 5^4). 
\end{equation*}
So there is no contribution to the finite places in Theorem \ref{hyperelliptic}.
In fact,  $C$ has potentially good
reduction everywhere. 
This was already observed by Bost, Mestre, and Moret-Bailly.

The different ideal  $\diff{F/\IQ}$ equals $\sqrt{5}\O{F}$. 
If $\omega_1 = 1$ and $\omega_2 = \sqrt{5}\zeta$, then 
\begin{equation*}
  \omega_1\O{F} + \omega_2 \diff{F/\IQ}^{-1} = 
\O{F}+\zeta\O{F} = \O{K}.
\end{equation*}
The period matrix
of $\O{K}$ can be computed using Remark \ref{rem:periodmatg2}
with $\theta = (5+\sqrt 5)/2$,
\begin{equation*}
  \tau_1 = \sqrt{5}\zeta,\quad\text{and}\quad
  \tau_2 = -\sqrt{5}\zeta^3
\end{equation*}
as
\begin{equation*}
  Z =  \left(
  \begin{array}{cc}
    \sqrt{5}^{-1}(\zeta - \zeta^3) & -1- \zeta\frac{1+\sqrt{5}}{2}
    \\
-1- \zeta\frac{1+\sqrt{5}}{2} & 2\sqrt{5}\zeta + \frac{5+\sqrt{5}}{2}
  \end{array}
\right).
\end{equation*}
We observe
$\det\imag{Z} = {\sqrt{5}}/4$ and use a computer to approximate
\begin{equation*}
-\frac{1}{10}\log(|\chi_{10}(Z)|\det\imag{Z}^{5})  = 0.246738390651711\ldots.
\end{equation*}
We add $-\log(2^{4/5}\pi)$ in accordance with Theorem 
\ref{hyperelliptic} and find that the sum approximates
(\ref{eq:hJC1}) up-to the displayed digits.


\bigskip
 \noindent{\bf Example 2.}
The second example concerns the new curve $C$
\begin{equation*}
y^2 = -103615x^6 - 41271x^5 + 17574x^4 + 197944x^3 + 67608x^2 - 103680x - 40824.
\end{equation*}
The endomorphism ring of the jacobian $\jac{C}$ has complex
multiplication by the ring of algebraic integers in 
$K= \IQ(\sqrt{-61+6\sqrt{61}})$. The real quadratic subfield of $K$ is
$F= \IQ(\sqrt{61})$. We have $\Delta_K=61^3$ and $\Delta_F= 61$, so
the conductor of $K$ is $f_K = 61$. 
Now $\diff{F/\IQ} = \sqrt{61}\O{F}$. 
Let $\chi :(\IZ/61\IZ)^\times \rightarrow\IC^\times$
 be the character of order $4$ with $\chi(2) = i$, observe that $2$
 generates
$(\IZ/61\IZ)^\times$. 
Then
 \begin{equation*}
 \sum_{m=1}^{60}\chi(m)m = -61(1 - i) 
 \end{equation*}
and so
\begin{equation}
\label{eq:hCex2}
h(\jac{C}) \dot{=} \frac 12 \log 61 - \frac 12 
\sum_{m=1}^{60} \real{\chi(m)(1+i)}\log\Gamma\left(\frac{m}{61}\right)
= 0.2688651723313\ldots
\end{equation}
by (\ref{eq:hAalt}). 

The Igusa invariants satisfy
\begin{equation*}
\frac{J_8^5}{J_{10}^4}=  -2^{40}\cdot 3^{-91}\cdot
5^{-48}\cdot 41^{-48}\cdot 643^5 \cdot 1871^5
\cdot 19780292330676250264630993^5, 
\end{equation*}
\begin{equation}
\label{ex2:char3}
\frac{J_6^5}{J_{10}^3}= 2^{25}\cdot 3^{-72}\cdot 5^{-36}
\cdot 7^5\cdot 41^{-36} \cdot 487^5\cdot 3449^5
\cdot 3467^5 \cdot 42488533591199^5,
\end{equation}
and
\begin{equation}\label{ex2:chargreater3}
\frac{J_2^5}{J_{10}}=- 2^{25}\cdot 3^{-19}\cdot 5^{-12}
\cdot 7^{15}\cdot 41^{-12}\cdot 39079^5. 
\end{equation}
The quotient (\ref{ex2:char3}) yields the contribution of $3$ to
the Faltings height and (\ref{ex2:chargreater3}) the contribution of
$5$ and $41$. Explicitly, the
finite contribution to $h(\jac{C})$ as in Theorem \ref{hyperelliptic}
is 
\begin{equation}
\label{eq:finitepartex2}
\frac 25 \log 3 + \frac 15 \log 5 + \frac 15 \log 41. 
\end{equation}
Our curve has potentially good reduction away from $3,5,$ and $41$. 

We fix roots $\tau_1,\tau_2\in\IH$ of $x^4 - 61x^3 + 6039x^2 - 137677x
+ 889319$. They are suitable diagonal elements as in Remark \ref{rem:periodmatg2} can be used to
construct a period matrix $Z$ with $\theta = (61+\sqrt{61})/2$. We approximate
\begin{equation*}
-\frac{1}{10}\log(|\chi_{10}(Z)|\det\imag{Z}^{5})  = 0.464065891333779\ldots.
\end{equation*}
We add (\ref{eq:finitepartex2}) and $-\log(2^{4/5}\pi)$ from
Theorem \ref{hyperelliptic} to this value and see that the resulting value
approximates
(\ref{eq:hCex2}) well.

\bigskip
 \noindent{\bf Example 3.} Our final example has bad reduction
 above $2$. Let $C$ be given by
 \begin{equation*}
 y^2 = -x^5 + 3x^4 + 2x^3 - 6x^2 - 3x + 1.
 \end{equation*}
The endomorphism ring of  $\jac{C}$ is the ring of integers in $K
 = \IQ(\sqrt{-2+\sqrt{2}})$ which contains $F=\IQ(\sqrt 2)$. 
We have
 $\Delta_K = 2^{11}$ and $\Delta_F=2^3,$
 as well as $f_K = 2^4$. 
We must take slightly more care when finding $\chi$ as
$(\IZ/16\IZ)^\times\cong \IZ/4\IZ\times \IZ/2\IZ$ is not cyclic and
 admits $4$ characters of order $4$. 
The kernel of $\chi$ we are interested in corresponds
to the fixed field of $K$ in the number field generated by a root of
 unity of order $16$. 
The non-trivial element in $\ker\chi \subseteq (\IZ/16\IZ)^\times$ is represented either by
 $7,9,$ or $-1$. However, $a^2\equiv 1\text{ or }9\mod 16$ if $a$ is
 odd. This rules our $9$ as a representative because $K/\IQ$ is cyclic
 of order $4$. Moreover, $-1$ is also impossible because it represents
 complex conjugation in the Galois group. This leaves  $7$,
\textit{i.e.} $\chi(7)=1$.
We must have $\chi(15)=-1$ and $\chi(9) = \chi(7\cdot 15)=-1$. 
 Again
 up-to complex conjugation there are at most $2$ choices for $\chi$. 
As $\chi(3) = \chi(3\cdot 7)=\chi(5)$  one choice is
 \begin{equation*}
 \begin{array}{c|cccccccc}
m & 1 & 3 & 5 & 7 & 9 & 11 & 13 & 15\\
\hline
\chi(m) & 1 & i  & i & 1 & -1  & -i  & -i  & -1 \\
\end{array}
 \end{equation*}
Thus
\begin{equation*}
\sum_{m=1}^{15} \chi(m)m = -16 (1+i)
\end{equation*}
and so 
\begin{equation*} 
h(\jac{C})\dot{=}\log 4 + \frac 12 \log\left(\frac{\Gamma\left(\frac{9}{16}\right)
\Gamma\left(\frac{11}{16}\right)
\Gamma\left(\frac{13}{16}\right)
\Gamma\left(\frac{15}{16}\right)}{\Gamma\left(\frac{1}{16}\right)
\Gamma\left(\frac{3}{16}\right)
\Gamma\left(\frac{5}{16}\right)
\Gamma\left(\frac{7}{16}\right)}
\right)
\end{equation*}
by (\ref{eq:hAalt}). 
Numerically, we find
\begin{equation}
\label{eq:hex3}
h(\jac{C})\dot{=}  -1.2016102497487\ldots.
\end{equation}

The Igusa invariants satisfy
\begin{equation*}
\frac{J_8^5}{J_{10}^4} = -2^{-24}\cdot 3^{10}\cdot
2029^5,\quad
\frac{J_6^5}{J_{10}^3} = 2^{-8} \cdot 3^5 \cdot 47^5,
\quad\text{and}\quad
\frac{J_2^5}{J_{10}} = 2^4\cdot 3^{15}.
\end{equation*}
So only $2$ contributes to the finite part of the height
in Theorem \ref{hyperelliptic}. In fact, $C$ has potentially good
reduction outside of $2$. The contribution to the finite part is 
\begin{equation*}
\frac{1}{10}\log 2. 
\end{equation*}
We can take 
\begin{equation*}
\tau_1 = 2\sqrt{-2+\sqrt 2}\sqrt{2}\quad\text{and}\quad
\tau_2 = 2\sqrt{-2-\sqrt 2}\sqrt{2}
\end{equation*}
to construct $Z$, now with $\theta = (2+\sqrt{2})/2$ and find
\begin{equation*}
-\frac{1}{10}\log(|\chi_{10}(Z)|\det\imag{Z}^{5})  = 0.428322662492607\ldots.
\end{equation*}
We must add
$(\log 2)/10$ to this value to compensate for bad reduction and
 $-\log(2^{4/5}\pi)$ due to the normalisation of the archimedean places. We end up with a good match with (\ref{eq:hex3}).
 

\end{appendix}

\bibliographystyle{amsplain}
\bibliography{literature}

\def\cprime{$'$}
\providecommand{\bysame}{\leavevmode\hbox to3em{\hrulefill}\thinspace}
\providecommand{\MR}{\relax\ifhmode\unskip\space\fi MR }
\providecommand{\MRhref}[2]{%
  \href{http://www.ams.org/mathscinet-getitem?mr=#1}{#2}
}
\providecommand{\href}[2]{#2}
\begin{thebibliography}{10}

\bibitem{Aut06}
P.~Autissier, \emph{Hauteur de {F}altings et hauteur de {N}\'eron-{T}ate du
  diviseur th\^eta}, Compos. Math. \textbf{142} (2006), no.~6, 1451--1458.

\bibitem{Badzyan}
A.~I. Badzyan, \emph{The {E}uler-{K}ronecker constant}, Mat. Zametki
  \textbf{87} (2010), no.~1, 35--47.

\bibitem{CAV}
C.~Birkenhake and H.~Lange, \emph{{C}omplex {A}belian {V}arieties}, Springer,
  2004.

\bibitem{BG}
E.~Bombieri and W.~Gubler, \emph{{H}eights in {D}iophantine {G}eometry},
  Cambridge University Press, 2006.

\bibitem{NeronModels}
S.~Bosch, W.~L{\"u}tkebohmert, and M.~Raynaud, \emph{N\'eron models},
  Ergebnisse der Mathematik und ihrer Grenzgebiete (3), vol.~21,
  Springer-Verlag, Berlin, 1990.

\bibitem{BMM90}
J.-B. Bost, J.-F. Mestre, and L.~Moret-Bailly, \emph{Sur le calcul explicite
  des ``classes de {C}hern'' des surfaces arithm\'etiques de genre {$2$}},
  Ast\'erisque (1990), no.~183, 69--105, S{\'e}minaire sur les Pinceaux de
  Courbes Elliptiques (Paris, 1988).

\bibitem{CU:2005}
L.~Clozel and E.~Ullmo, \emph{\'{E}quidistribution de mesures alg\'ebriques},
  Compos. Math. \textbf{141} (2005), no.~5, 1255--1309.

\bibitem{Cohen:II}
H.~Cohen, \emph{Number {T}heory {V}olume {II}: {A}nalytic and {M}odern
  {T}ools}, Springer, 2007.

\bibitem{Cohen:equi}
P.B. Cohen, \emph{Hyperbolic equidistribution problems on {S}iegel 3-folds and
  {H}ilbert modular varieties}, Duke Math. J. \textbf{129} (2005), no.~1,
  87--127.

\bibitem{ColmezAnnals}
P.~Colmez, \emph{P\'eriodes des vari\'et\'es ab\'eliennes \`a multiplication
  complexe}, Ann. of Math. (2) \textbf{138} (1993), no.~3, 625--683.

\bibitem{Colmez}
\bysame, \emph{Sur la hauteur de {F}altings des vari{\'e}t{\'e}s ab{\'e}liennes
  {\`a} multiplication complexe}, Compositio Math. \textbf{111} (1998), no.~3,
  359--368.

\bibitem{deJongNoot:91}
J.~de~Jong and R.~Noot, \emph{Jacobians with complex multiplication},
  Arithmetic algebraic geometry ({T}exel, 1989), Progr. Math., vol.~89,
  Birkh\"auser Boston, Boston, MA, 1991, pp.~177--192.

\bibitem{DelMum}
P.~Deligne and D.~Mumford, \emph{The irreducibility of the space of curves of
  given genus}, Inst. Hautes {\'E}tudes Sci. Publ. Math. \textbf{36} (1969),
  75---109.

\bibitem{Faltings:ES}
G.~Faltings, \emph{{E}ndlichkeitss{\"a}tze f{\"u}r abelsche {V}ariet{\"a}ten
  {\"u}ber {Z}ahlk{\"o}rpern}, Invent. Math. \textbf{73} (1983), 349--366.

\bibitem{Font85}
J.-M. Fontaine, \emph{Il n'y a pas de vari\'et\'e ab\'elienne sur
  $\mathbb{Z}$.}, Invent. Math. \textbf{81} (1985), no.~3, 515--538.

\bibitem{Goren97}
E.Z. Goren, \emph{On certain reduction problems concerning abelian surfaces},
  Manusc. Math. \textbf{94} (1997), no.~1, 33--43.

\bibitem{GorenLauter:06}
E.Z. Goren and K.E. Lauter, \emph{Evil primes and superspecial moduli}, Int.
  Math. Res. Not. (2006), Art. ID 53864, 19.

\bibitem{GoLau}
\bysame, \emph{Class invariants for quartic {CM} fields}, Ann. Inst. Fourier
  (Grenoble) \textbf{57} (2007), no.~2, 457--480.

\bibitem{GriffithsHarris}
P.~Griffiths and J.~Harris, \emph{Principles of algebraic geometry},
  Wiley-Interscience [John Wiley \& Sons], New York, 1978.

\bibitem{hab:junit}
P.~Habegger, \emph{{S}ingular {M}oduli that are {A}lgebraic {U}nits}, {\tt
  arXiv:1402.1632}.

\bibitem{IKO}
T.~Ibukiyama, T.~Katsura, and F.~Oort, \emph{Supersingular curves of genus two
  and class numbers.}, Compositio Math. \textbf{57} (1986), no.~2, 127--152.

\bibitem{Igusa:60}
J.~Igusa, \emph{Arithmetic variety of moduli for genus two}, Ann. of Math. (2)
  \textbf{72} (1960), 612--649.

\bibitem{IK}
H.~Iwaniec and E.~Kowalski, \emph{Analytic number theory}, American
  Mathematical Society Colloquium Publications, vol.~53, American Mathematical
  Society, 2004.

\bibitem{Kling}
H.~Klingen, \emph{Introductory lectures on {S}iegel modular forms}, Cambridge
  Studies in Advanced Mathematics, vol.~20, Cambridge University Press,
  Cambridge, 1990.

\bibitem{Liu:stables}
Q.~Liu, \emph{Courbes stables de genre {$2$} et leur sch\'ema de modules},
  Math. Ann. \textbf{295} (1993), no.~2, 201--222.

\bibitem{Liu:cond}
\bysame, \emph{Conducteur et discriminant minimal de courbes de genre {$2$}},
  Compositio Math. \textbf{94} (1994), no.~1, 51--79.

\bibitem{Liu3}
\bysame, \emph{Algebraic geometry and arithmetic curves}, Oxford Graduate Texts
  in Mathematics, vol.~6, Oxford University Press, Oxford, 2002, Translated
  from the French by Reinie Ern{\'e}, Oxford Science Publications.

\bibitem{Lock}
P.~Lockhart, \emph{On the discriminant of a hyperelliptic curve}, Trans. Amer.
  Math. Soc. \textbf{342} (1994), no.~2, 729--752.

\bibitem{MV}
P.~Michel and A.~Venkatesh, \emph{The subconvexity problem for {${\rm GL}_2$}},
  Publ. Math. Inst. Hautes \'Etudes Sci. (2010), no.~111, 171--271.

\bibitem{MB:Skolem}
L.~Moret-Bailly, \emph{Probl\`emes de {S}kolem sur les champs alg\'ebriques},
  Compositio Math. \textbf{125} (2001), no.~1, 1--30.

\bibitem{Mum2}
D.~Mumford, \emph{Tata lectures on theta. {II}}, Progress in Mathematics,
  vol.~43, Birkh\"auser Boston, Inc., Boston, MA, 1984.

\bibitem{NT}
Y.~Nakkajima and Y.~Taguchi, \emph{A generalization of the {C}howla-{S}elberg
  formula}, J. Reine Angew. Math. \textbf{419} (1991), 119--124.

\bibitem{NamikawaUeno}
Y.~Namikawa and K.~Ueno, \emph{The complete classification of fibres in pencils
  of curves of genus two}, Manuscripta Math. \textbf{9} (1973), 143--186.

\bibitem{Neukirch}
J.~Neukirch, \emph{Algebraic number theory}, Grundlehren der Mathematischen
  Wissenschaften [Fundamental Principles of Mathematical Sciences], vol. 322,
  Springer-Verlag, Berlin, 1999.

\bibitem{Obus}
A.~Obus, \emph{On {C}olmez's product formula for periods of {CM}-abelian
  varieties}, Math. Ann. \textbf{356} (2013), no.~2, 401--418.

\bibitem{Paz}
F.~Pazuki, \emph{D{\'e}compositions en hauteurs locales},
  http://arxiv.org/abs/1205.4525.

\bibitem{Paz2}
\bysame, \emph{Theta height and {F}altings height}, Bull. Soc. Math. France
  \textbf{140} (2012), no.~1, 19--49.

\bibitem{Paz3}
\bysame, \emph{Minoration de la hauteur de {N}\'eron--{T}ate sur les surfaces
  ab\'eliennes.}, Manusc. Math. \textbf{142} (2013), 61--99.

\bibitem{PilaTsimerman}
J.~Pila and J.~Tsimerman, \emph{The {A}ndr\'e-{O}ort conjecture for the moduli
  space of abelian surfaces}, Compos. Math. \textbf{149} (2013), no.~2,
  204--216.

\bibitem{PilaTsimerman:AxLAg}
\bysame, \emph{Ax-{L}indemann for {$\mathcal A_g$}}, Ann. of Math. (2)
  \textbf{179} (2014), no.~2, 659--681.

\bibitem{Saito:Duke88}
T.~Saito, \emph{Conductor, discriminant, and the {N}oether formula of
  arithmetic surfaces}, Duke Math. J. \textbf{57} (1988), no.~1, 151--173.

\bibitem{Saito:Comp89}
\bysame, \emph{The discriminants of curves of genus {$2$}}, Compositio Math.
  \textbf{69} (1989), no.~2, 229--240.

\bibitem{Schoof03}
R.~Schoof, \emph{Abelian varieties over cyclotomic fields with good reduction
  everywhere}, Math. Ann. \textbf{325} (2003), no.~3, 413--448.

\bibitem{SerTat}
J.-P. Serre and J.~T. Tate, \emph{Good reduction of abelian varieties}, Ann. of
  Math. (2) \textbf{88} (1968), 492--517.

\bibitem{Shimura}
G.~Shimura, \emph{Abelian varieties with complex multiplication and modular
  functions}, Princeton Mathematical Series, vol.~46, Princeton University
  Press, 1998.

\bibitem{SPA}
L.~Szpiro (ed.), \emph{S\'eminaire sur les pinceaux arithm\'etiques: la
  conjecture de {M}ordell}, Soci\'et\'e Math\'ematique de France, Paris, 1985,
  Papers from the seminar held at the {\'E}cole Normale Sup{\'e}rieure, Paris,
  1983--84, Ast{\'e}risque No. 127 (1985).

\bibitem{Ueno}
K.~Ueno, \emph{Discriminants of curves of genus {$2$} and arithmetic surfaces},
  Algebraic geometry and communtative algebra, Kinokuniya \textbf{II} (1988),
  749--770.

\bibitem{vdGeer}
G.~van~der Geer, \emph{Hilbert modular surfaces}, Ergebnisse der Mathematik und
  ihrer Grenzgebiete (3) [Results in Mathematics and Related Areas (3)],
  vol.~16, Springer-Verlag, Berlin, 1988.

\bibitem{vanWamelen}
P.~van Wamelen, \emph{Examples of genus two {CM} curves defined over the
  rationals}, Math. Comp. \textbf{68} (1999), no.~225, 307--320.

\bibitem{vanWamelen2}
\bysame, \emph{Proving that a genus {$2$} curve has complex multiplication},
  Math. Comp. \textbf{68} (1999), no.~228, 1663--1677.

\bibitem{Vojta:semiabII}
P.~Vojta, \emph{Integral points on subvarieties of semiabelian varieties.
  {II}}, Amer. J. Math. \textbf{121} (1999), no.~2, 283--313.

\bibitem{DuFrIw02}
H.~Iwaniec W.~Duke, J.B.~Friedlander, \emph{The subconvexity problem for artin
  l-functions}, Invent. Math. \textbf{149} (2002), 489--577.

\bibitem{Washington}
L.C. Washington, \emph{{I}ntroduction to {C}yclotomic {F}ields}, Springer,
  1982.

\bibitem{Yang10}
T.~Yang, \emph{The {C}howla-{S}elberg formula and the {C}olmez conjecture},
  Canad. J. Math. \textbf{62} (2010), no.~2, 456--472.

\bibitem{zhang:equicm}
S.~Zhang, \emph{Equidistribution of {CM}-points on quaternion {S}himura
  varieties}, Int. Math. Res. Not. (2005), no.~59, 3657--3689.

\end{thebibliography}

\end{document}